\documentclass[11pt, a4paper]{amsart}

\usepackage{amsmath, amsthm, amssymb,fullpage}
\usepackage{amsmath}
\usepackage{amsfonts}
\usepackage{amssymb}
\usepackage{amsthm}
\usepackage{enumitem} 

\usepackage{color}
\usepackage{xcolor}
\usepackage{enumitem}
\usepackage{hyperref}

\newcommand{\Fc}{\mathcal{F}}

\DeclareMathOperator{\Psh}{PSH}
\DeclareMathOperator{\Dsh}{DSH}
\DeclareMathOperator{\psh}{Psh}

\theoremstyle{plain}
\newtheorem{theorem}{Theorem}[section]

\newtheorem{lemma}[theorem]{Lemma}
\newtheorem{proposition}[theorem]{Proposition}

\theoremstyle{definition}
\newtheorem{definition}[theorem]{Definition}

\theoremstyle{remark}
\newtheorem{remark}[theorem]{Remark}

\numberwithin{equation}{section}

\renewcommand{\P}{\mathbb P}
\newcommand{\C}{\mathbb C}

\newcommand{\lam}{\lambda}
\newcommand{\eps}{\epsilon}

\newcommand{\Cc}{\mathcal C}

\usepackage{bbold}
\newcommand{\1}{\mathbb 1}

\newcommand{\D}{\mathbb D}
\newcommand{\N}{\mathbb N}

\newcommand{\G}{\mathbb G}
\usepackage{tikz-cd}

\DeclareMathOperator{\dist}{dist}

\newcommand{\ddc}{{dd^c}}

\usepackage{enumitem}
\usepackage{hyperref}

\begin{document}

\hyphenpenalty=10000

\title[]{H\"older continuity and laminarity of the Green currents
for H\'enon-like maps}

\begin{author}[F.~Bianchi]{Fabrizio Bianchi}
\address{Dipartimento di Matematica, Università di Pisa, Largo Bruno Pontecorvo 5, 56127 Pisa, Italy}
\email{fabrizio.bianchi@unipi.it}
\end{author}

\begin{author}[T.C.~Dinh]{Tien-Cuong Dinh}
\address{National University of Singapore, Lower Kent Ridge Road 10,
Singapore 119076, Singapore}
\email{matdtc$@$nus.edu.sg }
\end{author}

\begin{author}[K.~Rakhimov]{Karim Rakhimov}
\address{National University of Singapore, Lower Kent Ridge Road 10,
Singapore 119076, Singapore}
\curraddr{V.I. Romanovskiy  Institute of Mathematics of Uzbek Academy of Sciences,  Tashkent, Uzbekistan}
\email{karimjon1705@gmail.com}
\end{author}

\begin{abstract}
Under a natural assumption on the dynamical degrees,
we prove that the Green currents associated to
any
H\'enon-like map in any dimension
have H\"older continuous super-potentials, i.e.,
give H\"older continuous linear functionals
on suitable spaces of forms and currents.
As a consequence, the unique measure of maximal entropy is the Monge-Amp\`ere of a
H\"older continuous 
plurisubharmonic function
and has strictly positive Hausdorff dimension.
Under the same assumptions,
we also prove that the
Green
currents are woven. When they are of bidegree $(1,1)$, they are laminar.
In particular, our results generalize results  known until now only in algebraic settings, or in dimension 2. 
\end{abstract}

\maketitle

\setcounter{secnumdepth}{3}
\setcounter{tocdepth}{1}
\tableofcontents

\section{Introduction}

H\'enon maps are among the most studied dynamical systems.
They were introduced by Michel Hénon in the real setting as a simplified model
of the Poincaré section for the Lorenz model, see, e.g., \cite{BC91,Henon76}.
They are also actively
studied in the complex setting, where complex analysis offers additional powerful tools,
see the fundamental work of
Bedford-Lyubich-Smille and Fornaess-Sibony \cite{BLS93,BLS93b,BS91,BS92,FS92, Sibony99}. 
As an example, a unique measure of maximal entropy $\mu$ was introduced by Sibony, as the intersection $\mu = T^+\wedge T^-$ of two positive closed currents with H\"older continuous potentials.
These currents can be seen as the accumulation of the iterates of manifolds
under backward and forward iteration, respectively \cite{BS91}. The regularity of their potentials plays an important role in the quantification of the speed of
such convergences
\cite{Dinh05,DS14rigidity}.
These two currents are also \emph{laminar}, namely they can be  nicely approximated by 
the integration on (pieces) of complex curves which do not intersect
and the measure
$\mu$
has a local
product structure
\cite{BLS93,Dujardin04}.
Such measure 
enjoys remarkable statistical properties, see for instance \cite{BLS93,BLS93b,BD24,Dinh05,Wu22henon}.

\medskip

A natural generalization of H\'enon maps in any dimension is given by the so-called \emph{Hénon-Sibony maps}.
These are polynomial
automorphisms $f$
of $\mathbb C^k$ such that the indeterminacy sets of 
the extensions to $\mathbb P^k$
of $f$ and $f^{-1}$ are non-empty and disjoint \cite{Sibony99}. Although their study is highly more technical, due to the need to work with
cycles and currents of intermediate dimensions,
the main properties of 
H\'enon maps mentioned above have been successfully generalized to this setting.
In particular, the unique measure of maximal entropy
can still be obtained as the intersection of two
Green currents $T^+$ and $T^-$ 
of complementary bidegree.
These currents have H\"older continuous \emph{super-potentials}
 \cite{DS09super,Sibony99}. 
Roughly speaking,
they can be seen as H\"older continuous maps on  a
suitable
space of forms of complementary degree.
Moreover, these currents are 
\emph{woven} \cite{Dinh05suites,DS16}.
This is a slightly weaker notion than the
laminarity,
where the approximating analytic sets can a priori intersect, see
\cite{Dinh05suites} and Definition \ref{d:woven}.
These properties of the Green currents were also investigated in the setting of automorphisms of K{\"a}hler manifolds \cite{BD23kahler, Cantat01, dThD, DS05Green,  Wu21kahler}
and
for generic birational self-maps
of $\mathbb P^k$
\cite{BD05energy,dThV10,dThV24,Dujardin06,Vigny15}.

\medskip

It is natural to ask how much of the above stays true when leaving the algebraic setting (for instance, when considering small perturbations of H\'enon-Sibony maps, which destroy the algebraic structure, but keep the global geometric properties).
This question was first addressed by
Dujardin \cite{Dujardin04},
who
systematically
studied the class
of so-called \emph{H\'enon-like maps} in two dimensions,  introduced by Hubbard and Oberste-Vorth in
\cite{HOV95}.
He proved that such maps still enjoy many of the properties
of H\'enon maps, starting from 
a unique measure of maximal entropy with a local product structure.

\medskip

The goal of this work is to close this circle by proving that, under a natural  and necessary
assumption on the  dynamical degrees of the map (which is automatically satisfied in the above algebraic setting and in dimension 2), every H\'enon-like
map in any dimension
has the same
strong regularity and 
laminarity properties
of their algebraic counterparts. 

\medskip

Let us now be more precise and state our main result. 
Given integers $1\leq p < k$ and
open bounded 
convex domains $M\Subset \mathbb C^p$ and $N\Subset \mathbb C^{k-p}$,
a \emph{H\'enon-like map}
is an invertible proper holomorphic map from a vertical subset to a horizontal subset of $M\times N$ which geometrically 
(but non-uniformly)
expands in $p$ directions and contracts in $k-p$
directions, see
\cite{DNS,DS1} and Definition \ref{d:hlm} below
for the
precise definition.
They should be thought of as the building blocks of more complicated systems.
They form an infinitely dimensional family, which contains maps which are not conjugated to Hénon-Sibony maps.
These systems admit a measure of maximal entropy $\mu$, which is the intersection 
$\mu= T^+\wedge T^-$
of a vertical Green current $T^+$ and of a horizontal Green current $T^-$,
 of bidimension 
 $(k-p,k-p)$ and
$(p,p)$ respectively.

\medskip

A natural assumption when studying higher dimensional dynamical systems is that
there exists an integer $q$ such that the \emph{dynamical degree} of order $q$ (describing the rate of growth of the volume of manifolds -- or more generally of the mass of positive closed currents -- of dimension $q$ by iteration)
strictly dominates all the other dynamical degrees.
 In our setting, precise definitions are given in \cite{BDR,DNS} and
Definition \ref{d:degrees-currents}.
 Because of the geometric properties of Hénon-like maps, our
main dynamical degree is that of order $p$, and can be seen as the rate of growth of the volume for the forward iteration of \emph{horizontal} $p$-dimensional manifolds, or equivalently for the backward iteration of \emph{vertical}
$(k-p)$-dimensional manifolds. We denote by $d\, (=d^+_p=d_{k-p}^-)$ this degree.
We then denote, for $0\leq s\leq p-1$,
by $d^{+}_{s}$ the rate of growth of  the volume for the
forward iteration of
horizontal $s$-dimensional manifolds, and 
by
$d^{-}_{s}$, for $0\leq s \leq k-p-1$,
the rate of growth  of the
volume of the backward  iteration of
vertical $s$-dimensional manifolds.
We proved in \cite{BDR} that the sequences
$\{d^+_{s}\}_{0\leq s \leq p}$ and
$\{d^-_{s}\}_{0\leq s \leq k-p}$ are monotone.
In particular,
as soon as we have
$d> \max (d_{p-1}^+, d_{k-p-1}^-)$, the main degree
$d$ strictly dominates all dynamical degrees.
The main result of this paper can be stated as follows.

\begin{theorem}\label{t:intro-all}
Let $f$ be a H\'enon-like map in any dimension
 as above
and assume that the main dynamical degree $d=d_p^+= d_{k-p}^-$
satisfies $d> \max (d_{p-1}^+, d_{k-p-1}^-)$.
Then
\begin{enumerate}
\item[{\rm (i)}]  the Green currents $T^\pm$ have H\"older continuous super-potentials;
\item[{\rm (ii)}]  the Green currents $T^\pm$ are woven;
\item[{\rm (iii)}]  $\mu=T^+\wedge T^-$ is a Monge-Amp\`ere measure with H\"older continuous potential and it has a positive Hausdorff dimension.
\end{enumerate}
If, furthermore, we have 
 $p=k-1$ (resp. $p=1$), 
then $T^-$ (resp. $T^+$) is laminar.
\end{theorem}

It was proved in \cite{DNS} that, under
the same assumptions of Theorem \ref{t:intro-all},
the Green currents have continuous super-potentials.
The first assertion of Theorem \ref{t:intro-all} 
is an improvement of that result. The proof is based
on
quantitative estimates for the action of $f_*$ and $f_*$
with respect to suitable norms, involving a precise control for the solution of the $dd^c$ equation for horizontal and vertical currents in $D$.
The third assertion is a consequence of the first and of a more general criterion to ensure that property, see Proposition \ref{p:MA-gen}. After some preliminary material presented in Section \ref{s:holder-general}, these two assertions are proved in Section \ref{s:holder}.

\medskip

The 
 most difficult
part of Theorem \ref{t:intro-all} 
consists in the laminarity statements for the Green currents.
A natural way to prove laminarity properties for $T^+$
is to exploit the convergence $d^{-n}[{f}^{-n} \Sigma]\to T^+$, where $\Sigma$
is a vertical analytic set of dimension $k-p$. By a result
of de Thélin \cite{dTh},
in order to show that $T^+$ is woven, one can verify that
\begin{equation}\label{eq:intro-condition}
\mathrm{volume} (\hat \Sigma_n)= O(\mathrm{volume} (\Sigma_n))
\quad \mbox{ as } 
 n\to \infty,
\end{equation}
where
$\Sigma_n :=[{f}^{-n} \Sigma]$,
$\hat \Sigma_n$
is a natural
lift of $\Sigma_n$ to the space $\D\times \G$,
and $\G$
is a suitable
Grassmannian.
In the algebraic settings above, this is achieved by means of 
cohomology arguments.
On the other hand,
estimating directly
the volume of $\hat \Sigma_n$
by means of the volume of $\Sigma_n$ is not possible in our 
non-algebraic and local 
setting, because of the lack of a meaningful Hodge theory.
In order to overcome this issue, we introduce 
the notion of \emph{shadow} for a non-necessarily closed current on $D\times \G$,
corresponding to
a geometrically ``correct dimensional" 
projection on $D$,
taking into account the possible defect of dimension (which cannot be detected by cohomology), see Section \ref{ss:shadows}. We then show 
that, if
\eqref{eq:intro-condition} does not hold for some $\Sigma$ as above, 
the extra volume of $\hat \Sigma_n$
must come from some part of $\hat \Sigma_n$
whose shadow is of smaller dimension (since the part 
whose shadow has full dimension is controlled by the dynamical degree $d$).
Using this, we 
can construct a sequence $S_n$ of horizontal positive closed currents on $D$ (related to the shadows of $\hat \Sigma_n$)
of dimension $l<p$ 
and satisfying $\|f^n_* (S_n)\|\gtrsim d^n\|S_n\|$. This leads to a contradiction with the assumption $d_l^+ \leq d_p^+ < d$, and completes the proof.

\medskip

\subsection*{Acknowledgements}

The authors would like to thank the National University of Singapore,
 the University of Lille, and the University of Pisa
 for the warm welcome and the excellent work conditions.

This project has received funding from 
the Programme
Investissement d’Avenir
(LabEx CEMPI /ANR-11-LABX-
0007-01, ANR QuaSiDy /ANR-21-CE40-0016, ANR PADAWAN /ANR-21-CE40-0012-01),
from the 
Ministry of Education, Singapore, through the grant
MOE-T2EP20120-0010, 
from
the MIUR Excellence Department Project awarded to the Department of Mathematics, University of Pisa, CUP I57G22000700001, and from 
the   government of Uzbekistan
 through the grant
 IL-5421101746.
The first
author is affiliated to the GNSAGA group of INdAM.

\section{Horizontal-like maps and dynamical degrees}

In this section, we recall the definition of a Hénon-like map and its dynamical degrees, 
as well as the main properties of these maps that we will need in the sequel.

\medskip

Let $p$ and $k$ 
be integers with
$1\le p\le k-1$.
Let $M\Subset\mathbb{C}^p$ and $ N\Subset\mathbb{C}^{k-p}$ be open bounded
convex sets.
Denote by $\pi_M$ and $\pi_N$ the natural projections of $D:=M\times N$ to $M$ and $N$, respectively. A subset $E\subset D$ is  \textit{horizontal} (resp. \textit{vertical}) if $\pi_{N}(E)\Subset N$ (resp. $\pi_{M}(E)\Subset M$).  
The horizontal (resp. vertical) boundary of $D$ is the set
$\partial_h D:= \overline{M}\times \partial N$ (resp. $\partial_v D:=\partial M\times \overline{N}$).  
We also denote by $\pi_1$ and $\pi_2$ the natural projections of $D\times D$ to its first and second factors, respectively. In all this paper, the symbols $\lesssim$ and $\gtrsim$ denote inequalities that hold up to an implicit
multiplicative constant (often depending on the domains under consideration).

\begin{definition}\label{d:hlm}
    A \textit{horizontal-like map} $f$ of $D$ is a holomorphic map whose graph $\Gamma\subset D\times D$ is a (not necessarily connected) submanifold of $D\times D$ with pure dimension $k$  and such that
    \begin{enumerate}
             \item $\pi_1|_\Gamma$ is injective and $\pi_2|_\Gamma$ has finite fibers; and
        \item $\overline \Gamma$ does not intersect ${\partial_v D} \times \overline D$ and
    $\overline D \times  {\partial_h D}$.
       \end{enumerate}
        We say that $f$ is an \textit{invertible horizontal-like map}, or a \textit{H\'enon-like map}, if $\pi_2|_\Gamma$ is also injective.
\end{definition}

Observe
that $f$ is not 
defined on the whole 
$D$, but only on the vertical open subset $\pi_1 (\Gamma)$ of $D$. Similarly, the image of $f$ is the horizontal subset $\pi_2 (\Gamma)$ of $D$. We will write $D_{v,1}:= \pi_1 (\Gamma)$ and $D_{h,1}:= \pi_2 (\Gamma)$ in the following.
More generally, for all $n\geq 1$ we consider the iterate $f^n = f\circ \dots \circ f$ ($n$ times), which is
also a horizontal-like map.
We denote by $D_{v,n}$ and $D_{h,n}$
the domain and the image of $f^n$, respectively. Observe that the sequences $(D_{v,n})_{n\geq 1}$
and $(D_{h,n})_{n\geq 1}$ are decreasing.
 
\medskip

We fix in what follows a horizontal-like map $f$ of $D=M\times N$. We also fix
two 
convex open subsets $M'\Subset M$ and $N'\Subset N$. For simplicity,
we will always assume that $M'$ and $N'$ are
sufficiently close to
$M$ and $N$, respectively,
to always satisfy
\begin{equation}\label{e:hp-inclusion-1}
f^{-1}(D) \subset M' \times N
\quad \mbox{ and } \quad
f(D) \subset M \times N'.
\end{equation}

A \textit{horizontal} (resp. \textit{vertical}) current $S$ on {$M'\times N'$} 
is a current with horizontal (resp. {vertical}) support in $M'\times N'$.

\begin{definition}
\label{d:degrees-currents}
For every
$0\leq l\leq p$, the dynamical degree $d_l^+$ is defined by
$$d_l^+:=\limsup_{n\to\infty} \left(d^+_{l,n}\right)^{1/n}, \quad
\mbox{ where }
\quad
d^+_{l,n} :=\sup_{S} \|(f^n)_*(S)\|_{M'\times N}$$
and $S$ runs over the set 
of all horizontal positive closed currents of bidimension $(l,l)$ of mass 1 on $M'\times N'$.

Similarly,
for every $ 0\leq l\leq k-p$,
the dynamical degree $d_l^-$ 
is defined by
$$d_l^-:=\limsup_{n\to\infty} \left(d^-_{l,n}\right)^{1/n}, 
\quad
\mbox{ where }
\quad
d^-_{l,n} :=\sup_{R} \|(f^n)^*(R)\|_{M\times N'}$$
and $R$ runs over the set of all vertical positive closed currents of bi-dimension $(l,l)$ of mass 1 on $M'\times N'$.
\end{definition}
The dynamical degrees $d^+_l$ and $d^-_l$ are independent of the choice of $M'$ and $N'$
\cite{DNS}.
By
\cite[Proposition 4.2]{DS1},
both $d^+_p$ and $d^-_{k-p}$ are integers and we have $d^+_p=d^-_{k-p}=:d$.  
The integer $d$ is called \emph{the main dynamical degree} of $f$.
Moreover,  we have 
\begin{equation}\label{eq:dpn=d}
    d^n \lesssim d^+_{p,n} \lesssim d^n 
\quad 
\mbox{ and }
\quad 
d^n \lesssim
d^{-}_{{ k-p},n} \lesssim d^n \quad \mbox{ as }n\to\infty.
\end{equation}

When $f$ is invertible,
we have 
$d^+_l (f) = d^-_l (f^{-1})$ 
for all $0 \leq  l \leq p$ 
and $d^-_l (f) = d^+_l (f^{-1})$
for all $0 \leq l \leq  k - p$, 
where the degrees of $f^{-1}$
can be defined reversing the role of $f_*$ and $f^*$
and using the fact that $f^{-1}$ 
is a vertical-like map with $k - p$
expanding directions, see also \cite[Section 3]{DNS}.

\medskip

Every Hénon-like map admits
a
canonical horizontal
current $T^-$, 
of bidimension $(p,p)$, and a canonical
vertical
 current
 $T^+$,  of bidimension 
 $(k-p,k-p)$. These currents are given and characterized by the following result.
Recall that, by \cite[Theorem 2.1]{DS1}
(see also \cite[Proposition 2.7]{Dujardin04}),
for any horizontal current
$S$ of bidimension $(p,p)$ on $D$,
the \emph{slice measure}
 $\langle S,\pi_{M'}, z\rangle$
is well-defined for any $z\in M'$ and its mass, denoted by $\|S\|_h$, is independent of $z$
and is equal to
$\langle S, \pi_{M}^* (\Omega_M)\rangle$
for every smooth probability measure $\Omega_M$
with compact support in $M'$.
Similarly, 
for every vertical current $R$
of bidimension $(k-p,k-p)$ on $D$,
the 
slice measure
$\langle R,\pi_{N'}, w\rangle$
is well-defined for any $w\in N'$ and its mass, denoted by $\|R\|_v$ is independent of $w$ and equal to $\langle R, \pi_N^*(\Omega'_N)\rangle$
for every smooth probability measure $\Omega'_N$
with compact support in $N'$.

\begin{theorem}[\cite{DS1}] 
\label{t:greencurrents}
    Let $f$  be a
    Hénon-like
 map on $D$. Let $\Omega$ (resp. $\Theta$) be a
    closed
    vertical
    current of bidimension $(k-p,k-p)$ (resp. horizontal
    current of bidimension $(p,p)$)
    of slice mass  $1$
    in $D$. Then
    \[d^{-n}(f^n)^*\Omega\to T^+
    \quad \mbox { and } 
    \quad
    d^{-n}(f^n)_*\Theta\to T^-
    \quad \mbox{ as } n\to \infty,
    \] where
       $T^+$ (resp. $T^-$)
    is a positive closed 
    vertical 
    current of bidimension $(k-p,k-p)$ 
    (resp. horizontal
    current of bidimension $(p,p)$)
    of slice mass $1$. 
    
    The currents $T^+$ and $T^-$ are    independent of $\Omega$ and  $\Theta$.
    Moreover, we have
    \[d^{-2n}(f^n)^*\Omega\wedge(f^n)_*\Theta\to \mu:=T^+\wedge T^-.\]
\end{theorem}
We call $T^+$, $T^-$, and $\mu$ the
Green current of $f$, the Green current of $f^{-1}$, and the Green measure of $f$, respectively.

\section{Horizontal and vertical currents with H\"older continuous super-potentials}
\label{s:holder-general}

We fix in this section two integers $p$ and $k$ with $1\leq p \leq k-1$ and
bounded open convex 
domains $M''\Subset M'\Subset M\Subset \mathbb C^p$ and
$N''\Subset N'\Subset N \Subset \mathbb  C^{k-p}$.

\subsection{Super-potentials of positive closed vertical currents}
For any $0<l< \infty$,
we denote by $\|\cdot\|_{\mathcal C^l}$
the standard $\mathcal C^l$ norm on the space of differential forms of a given degree.
For every $0\leq j\leq k$,
we define 
a semi-norm $\mathcal C^{-l}$
and a semi-distance $\dist_{-l}$ on 
the space of horizontal
(resp. vertical)
$(j,j)$-currents
on $M'\times N'$
by
\[\|\Phi\|_{-l}:=\sup|\langle \Phi,\Omega\rangle|
\quad
\mbox{ and }
\quad 
\dist_{-l} (\Phi_1, \Phi_2):= \|\Phi_1 - \Phi_2\|_{-l},
\]
where the supremum is taken over  
all $(k-j,k-j)$-forms
$\Omega$ with vertical
support in $M''\times N$
(resp. horizontal
support in $M\times N''$)
and such that $\|\Omega\|_{\Cc^l}\leq 1$. 
Observe that the semi-norm $\|\cdot\|_{-l}$, 
and  as a consequence the semi-distance $\dist_{-l}$,
is well-defined for currents of order up to $l$.
Moreover, by definition, we have $\|\Phi\|_{-l}\leq \|\Phi\|_{-l'}$ for every $\Phi$ as above
and $0<l'<l<\infty$.

\medskip

Let now $\Psh_h(M'\times N')$
be the set of horizontal 
currents $\Phi$ on $M'\times N'$
of bidimension $(p,p)$ 
such that $\ddc \Phi\geq 0$. 
We denote by $\Dsh_h(M'\times N')$ 
the real space spanned 
by $\Psh_h(M'\times N')$, i.e., the set
\[
\Dsh_h(M'\times N') := \{\Phi_1-\Phi_2 \colon \Phi_1, \Phi_2 \in \Psh_h(M'\times N')\}.
\]
For every $\Phi \in \Dsh_h(M'\times N')$
we also set
\[\|\Phi\|_*
:=
\inf \{\|dd^c \Phi_1\| + \|dd^c \Phi_2\| \colon \Phi_1, \Phi_2 \in  \Psh_h(M'\times N') \mbox{ such that } \Phi=\Phi_1-\Phi_2\}.\]

\begin{remark}\label{r:def-vertical}
In a similar way, we can also consider the set $\Psh_v(M'\times N')$
of vertical 
currents $\Phi$ on $M'\times N'$
of bidimension $(k-p,k-p)$ 
such that $\ddc \Phi\geq 0$, and the real space  
$\Dsh_v(M'\times N')$ 
spanned 
by $\Psh_v(M'\times N')$. 
As the statements are completely analogous, we just consider $\Psh_h (M'\times N')$
and
$\Dsh_h (M'\times N')$ in the rest of this section.
\end{remark}

\begin{lemma}\label{l:l+2}
The following assertions hold for every $\Phi \in \Psh_h (M'\times N')$ and every $0<l<\infty$.
\begin{itemize}
\item[{\rm (i)}] 
There exists a positive constant $C$ independent of $\Phi$ and $l$
such that 
$\|\ddc\Phi\|_{-l-2}\leq C \|\Phi\|_{-l}$.
\item[{\rm (ii)}] 
For every compact subset $K\Subset M''\times N$ there exists a positive constant $C_{K,l}$
independent of $\Phi$
such that $\|\ddc \Phi\|_K\leq C_{K,l} \|\Phi\|_{-l}$.
\end{itemize}
\end{lemma}

Observe that the mass
$\|dd^c \Phi\|_K$ in the second item is well-defined since $dd^c \Phi$ is positive.

\begin{proof}
(i)
For every $\Phi\in\Psh_h(M'\times N')$ 
and every $(p-1,p-1)$-form
$\Omega$ with vertical support in $M''\times N$ we have
\[|\langle\ddc\Phi,\Omega\rangle|=|\langle\Phi,\ddc\Omega\rangle|\le \|\ddc\Omega\|_{\mathcal C^l} \|\Phi\|_{-l}
\lesssim 
\|\Omega\|_{\mathcal C^{l+2}} \|\Phi\|_{-l},
\] 
where the implicit constant in the last inequality is independent of $\Omega$, $\Phi$, and $l$. 
The first inequality above follows by observing that the support of $dd^c \Omega$
is contained in the support of $\Omega$, and hence in particular in $M''\times N$.

\medskip

(ii)
For every compact subset $K$ of $M''\times N$, there exists 
 a smooth function $\chi$ with vertical support in $M''\times N$
and such that $\1_{K}\leq \chi\leq 1$, where $\1_K$ is the characteristic function of $K$.
Since $\ddc \Phi \geq 0$, we have
\[
\|\ddc \Phi\|_{K} = \int_K \ddc \Phi \wedge  \omega^{p-1} \leq
\langle \ddc \Phi, \chi \omega^{p-1} \rangle
\leq \|\ddc \Phi\|_{-l-2} \|\chi \omega^p\|_{\mathcal C^{l+2}}
\lesssim 
\|\Phi\|_{-l} \|\chi \omega^p\|_{\mathcal C^{l+2}},
\]
where
 $\omega$ is the standard K\"{a}hler form of $\mathbb{C}^k$
and, by (i), the implicit constant in the last inequality is independent of $l$ and $\Phi$.
As  $\|\chi \omega^p\|_{\mathcal C^{l+2}}$
can be bounded by a constant depending only on $K$ and $l$, the assertion follows.
\end{proof}

\begin{definition}\label{d:holder-sp}
  Given  a constant $0<\alpha\leq 1$, a 
  positive closed $(p,p)$-current $T$ with vertical support in $M''\times N$ is said to
  \emph{have 
$(l,\alpha)$-H\"older continuous super-potentials}
  if $T$,  seen as a linear function on smooth  test $(k-p,k-p)$-forms with horizontal support 
in $M'\times N'$, 
extends to a function on  $\Psh_h(M'\times N')$ such that,
for every subset $\Fc \subset \Psh_h(M'\times N')$ which is bounded with respect
to
$\|\cdot \|_*$,
there exists a constant $C_{\mathcal{F}}$ such that 
\[|\langle T,\Phi_1-\Phi_2 \rangle|\le C_{\mathcal{F}}\dist_{-l}( \Phi_1,\Phi_2)^\alpha
\quad \mbox{for any}
\quad \Phi_1,\Phi_2\in\mathcal{F}.\]
\end{definition}

 Note that the function on $\mathcal F$
 defined by $T$ is seen as a \emph{super-potential} of $T$ in our point of view.

\begin{lemma}\label{l:smooth}
Every 
smooth positive closed $(p,p)$-form $T$ 
with vertical support in $M''\times N$
has $(l,1)$-H\"older continuous super-potentials
 for every $0<l<\infty$.
\end{lemma}

\begin{proof}
 We can assume that $\|T\|_{\mathcal C^{l}}\leq 1$.
Since $T$ is vertical in $M''\times N$,
it follows that, for every $\Phi_1, \Phi_2\in \Psh_h (M'\times N')$, we have
\[|\langle T,\Phi_1-\Phi_2\rangle|
\leq \|\Phi_1-\Phi_2\|_{- {l}} = \dist_{- {l}}
(\Phi_1,\Phi_2).\]
The assertion follows.
\end{proof}

\begin{remark}\label{r:interpolation-prelim}
By interpolation techniques
\cite{Triebel}, one can see
that
 if $T$ has  $(l,\alpha)$-H\"older continuous super-potentials for some $0<l<\infty$ and $0<\alpha<1$, then, up to slighly modifying
 $M'$ and $N'$,
 it has $(l',\alpha')$-H\"older continuous super-potentials for every $0<l'<\infty$ and  some $0<\alpha'<1$ depending on $l,l',\alpha$ but independent of $T$.
In the dynamical setting 
to which we will be interested, because of 
the geometric behaviour of our maps, 
such modification of the domains is not necessary, see Lemma \ref{l:interpolation:T+}.
 \end{remark}

We conclude this section with the following result, giving the regularity of the solutions of the $dd^c$ operator with respect to the norms introduced above.

\begin{theorem}\label{l:ddcpsi}
    Let $\Omega$ be a horizontal 
    positive closed current of bidimension $(p-1,p-1)$ on $M\times N''$
    and $l$ a positive number.  There exist
    a horizontal negative $L^1$ form $\Psi$ of bidimension $(p,p)$ on $M'\times N'$, 
  and a positive constant $c$
  independent of $\Omega$ 
  such that
  $dd^c\Psi=\Omega$
  on  $M'\times N'$
   and  we have
\begin{equation}\label{eq:normineq}
     \|\Psi\|_{M'\times N'}\le c\|\Omega\|_{M\times N''}
    \quad \mbox{ and } 
    \quad
     \|\Psi\|_{-l}\le c\|\Omega\|_{-l-1}.
 \end{equation} 
 \end{theorem}

Observe that the masses in the first inequality are well-defined since $\Psi$ and $\Omega$ are negative and positive, respectively.

\begin{proof}
Fix a convex open set $M^\star$ with $M'\Subset M^\star \Subset M$.
The existence of a horizontal negative $L^1$ form
$\Psi$ on $M^\star \times N$ satisfying $\ddc \Psi=\Omega$ and
$ \|\Psi\|_{M^\star\times N'}\lesssim \|\Omega\|_{M\times N''}$
follows from \cite[Theorem 2.7]{DNS}.
As $\Psi$ is defined by means of explicit kernels, and $M''\Subset M'\Subset M^\star$, this also gives the second estimate in \eqref{eq:normineq}.
\end{proof}

\subsection{Currents with H\"older continuous super-potentials and their intersections}

We denote in this section by $\psh_1 (D)$ the set
of plurisubharmonic (psh) functions $u$ on $D$  with $\|u\|_{L^1}\leq 1$. Observe that this norm defines a distance 
on $\psh_1 (D)$.

\medskip
We consider in the following 
positive closed
vertical $(p,p)$-currents on $M''\times N$, but it is clear that the definitions and statements apply also for 
horizontal $(k-p,k-p)$-currents on $M\times N''$.

\begin{definition}\label{d:moderate}
Let $T$ be a vertical $(p,p)$-current on $M''\times N$. 
We say that $T$ is \emph{moderate}
if, for any compact set $K\subset M'\times N'$, there exist
two 
constants $c>0$ and $\beta>0$ such that for any psh function $u\in \psh_1 (D)$
we have
$$\langle T_{|K}\wedge \omega^{k-p}, e^{\beta|u|}\rangle \leq c.$$
\end{definition}

\begin{lemma}
Let $T$ be a positive closed
vertical $(p,p)$-current 
on $M''\times N$
with $(l,\alpha)$-H\"older continuous
super-potentials, for some $0<l<\infty$ and $0<\alpha<1$. 
Then $T$ is moderate.
\end{lemma}
\proof
Fix a psh function $u$ on $D$. 
By reducing $D$
slightly
and subtracting from $u$ a constant, we
 can assume that $u\leq 0$. 
For every $m\in \mathbb N$, 
define $u_m:=\max(u,-m)$. Observe that $u-u_m$ is negative, supported on $\{u\leq -m\}$, and equal to $u+m$ on this set. 
It follows
that
 \begin{equation}\label{eq:skoda2}
|\langle \omega^k, u-u_m\rangle|
=\int_{\{u\leq -m\}} |u-u_m|
\lesssim 
e^{-\alpha_1 m} \int e^{\alpha_1 |u+m|}
\lesssim e^{-\alpha_1 m}.\end{equation}
for some positive
constant
$\alpha_1$,
where 
the integrals
are taken with respect to the Lebesgue measure and last inequality follows from 
Skoda's estimates 
\cite{S82}.

\medskip

Fix a smooth 
positive
closed
horizontal form $\Omega_K$ on $M'\times N'$
such that $\Omega_K\geq \omega^{k-p}$ on $K$, see for instance \cite[Lemma 2.3]{BDR}. Observe that we also have $\Omega_K \lesssim \omega^{k-p}$
 on $M''\times N'$.
Since $u$ and $u_m$ are psh and $\Omega_K$ is closed,
both $u\Omega_K$ and $u_m \Omega_{K}$ belong to $\Psh_h (M'\times N')$.
Fix 
a
$(p,p)$-form
$\Omega$
with vertical support in $M''\times N$
and such that $\|\Omega\|_{\mathcal C^l}\leq 1$.
As $|\Omega_K \wedge \Omega| \lesssim \omega^k$,
it follows from \eqref{eq:skoda2}
that
$|\langle \Omega_K\wedge \Omega, u-u_m\rangle| \lesssim e^{-\alpha_1 m}$, which implies that
$$\dist_{-l} (u\Omega_K, u_m\Omega_K) \lesssim e^{-\alpha_1 m}.$$
The assumption on the super-potentials of $T$,
together with the fact that $T$ and $u_m-u$ are positive,
implies that
$$
0\leq  \langle T\wedge \Omega_K, u_m- u \rangle
=\langle T, u_m\Omega_K \rangle - \langle T, u\Omega_K\rangle \lesssim e^{-\alpha_2 m}$$
for some positive 
constant
$\alpha_2$.
Hence, since $\Omega_K
{\geq}
\omega^{k-p}$,
we have
$$0\leq \langle T_{|K}\wedge \omega^{k-p}, u_m-u\rangle  \lesssim e^{-\alpha_2 m}
\quad \mbox{ for every } m\in \mathbb N.$$
The assertion follows by choosing any $\beta<\alpha_2$ in Definition \ref{d:moderate}.
\endproof

\begin{lemma}\label{l:T-holder-L1}
Let $T$ be 
a positive closed
vertical $(p,p)$-current 
on $M''\times N$
with $(l,\alpha)$-H\"older continuous
super-potentials, for some $0<l<\infty$ and $0<\alpha<1$.
Then we have
$$\dist_{-l} (u_1T,u_2T) \leq c \|u_1-u_2\|^\gamma_{L^1}
\quad \mbox{ for all } u_1, u_2, \in \psh_1 (D),$$
for some positive constants $c$ and $\gamma$ depending on $l$ and $\alpha$ but
independent of $u_1$ and $u_2$.
\end{lemma}

Observe that, as $T$ is vertical in $M''\times N$, it is also vertical in $M'\times N'$. Hence the distance in the statement is well-defined, see Remark \ref{r:def-vertical}.

\proof
Let $\rho$ be a smooth positive 
{radial}
function compactly supported on the unit ball of $\mathbb C^k$
and such that $\int \rho \omega^k =1$.
For every $\epsilon>0$,
set
$\rho_\epsilon (\cdot) := \epsilon^{-2k} \rho (\epsilon^{-1} \cdot)$. 
For every $u\in \psh_1 (D)$ and
every $\epsilon$ sufficiently small, the convolution
$u_{\epsilon} := u * \rho_\epsilon$
is well-defined on $M'\times N'$.
One can check that 
we have
\[
u \leq u_{\epsilon}
\quad \mbox{ and } \quad 
\|u_{\epsilon}-u\|_{L^1
{  (M'\times N')}
}\lesssim \epsilon |\log \epsilon|.\] 
Since the $u_{\eps}$'s and $u$ are psh,
it follows that for
every
fixed
positive closed smooth  
form $\Psi$ of bidimension $(p,p)$
and horizontal in $M'\times N'$,
all the products $u_{\epsilon}\Psi$ and $u \Psi$
belong to $\Psh_h (M'\times N')$. They also
form 
a bounded family with respect to $\|\cdot\|_*$.
Moreover, for every vertical $(p,p)$-form
$\Omega$
on $M''\times N$,
we have
\[
|\langle \Omega, 
(u_{\epsilon}-u)\Psi\rangle|
= |\langle \Omega \wedge \Psi, 
(u_{\epsilon}-u)\rangle|
\lesssim 
{  
\int_{M'\times N'}
(u_{\epsilon}-u) \omega^k
=}
\|u_{\epsilon}-u\|_{L^1
{  (M'\times N')}
},
\]
which implies that
$\dist_{-l}(u_{\epsilon}\Psi, u\Psi)\lesssim \epsilon |\log \epsilon|$.
Since $T$ has 
$(l,\alpha)$-H\"older continuous super-potentials, we deduce that
$$\langle T,(u_{\epsilon}-u)\Psi\rangle \lesssim 
\|u_\epsilon-u\|_{L^1 {  (M'\times N')}}^{\alpha}\lesssim
( \epsilon |\log \epsilon|)^{\alpha}\leq \epsilon^{\alpha'}$$
for every $\alpha'<\alpha$.
This and the inequality $u_{\epsilon}\geq u$ imply that
\begin{equation}\label{eq:T-est-1}
\dist_{-l}(u_{\epsilon} T, u T) \lesssim \epsilon^{\alpha'}.\end{equation}

\medskip

Let now consider $u_1$ and $u_2$ as in the statement. 
Denote by $u_{1,\epsilon}=u_1 * \rho_\epsilon$ and $u_{2,\epsilon}= u_2 * \rho_\epsilon$ 
the regularizations by convolution
of $u_1$ and $u_2$ respectively. It follows from the above and \eqref{eq:T-est-1} that there exists $\alpha'>0$ such that
\begin{equation}\label{eq:T-est-1b}
\dist_{-l}(u_{1,\epsilon} T, u_1 T) \lesssim \epsilon^{\alpha'}
\quad
\mbox{ and }
\quad
\dist_{-l}(u_{2,\epsilon} T, u_2 T) \lesssim \epsilon^{\alpha'}.
\end{equation}
We now get a similar estimate for $\dist_{-l}(u_{1,\epsilon} T, u_{2,\epsilon}T)$.
By the definition of
$u_{1,\epsilon}$
and $u_{2,\epsilon}$,
we
have
$$\|u_{1,\epsilon} -u_{2,\epsilon}\|_{\mathcal C^l}
=
\|
(u_1-u_2) * \rho_\epsilon
\|_{\mathcal C^l}
\lesssim \epsilon^{-2k-l} \|u_1-u_2\|_{L^1}.$$
Again by the fact that $T$ has $(l,\alpha)$-H\"older continuous super-potentials, we deduce that
\begin{equation}\label{eq:T-est-2}
\dist_{-l}(u_{1,\epsilon} T, u_{2,\epsilon}T) \lesssim [\epsilon^{-2k-l} \|u_1-u_2\|_{L^1}]^{\alpha}.\end{equation}
Together with \eqref{eq:T-est-1b}, this gives
\[
\dist_{-l}(u_{1} T, u_{2}T)
\lesssim \epsilon^{\alpha'} + 
[\epsilon^{-2k-l} \|u_1-u_2\|_{L^1}]^{\alpha}
\quad \mbox{ for every } \alpha'<\alpha.\]
Choosing $\epsilon := {  c_0}
\|u_1-u_2\|_{L^1}^{1/(1+2k+l)}$
{for some sufficiently small constant $c_0$},
we see that
\[
\dist_{-l}(u_{1} T, u_{2}T) \lesssim
\|u_1-u_2\|_{L^1}^{\alpha'/(1+2k+l)}.
\]
The assertion follows.
\endproof

\begin{proposition}\label{p:MA-gen}
Let $T_v$
be a positive closed vertical $(p,p)$-current in $M'' \times N$ and 
$T_h$ a positive closed horizontal $(k-p,k-p)$-current in $M\times N''$.  Assume that both $T_v$ and $T_h$ have
$(l, \alpha)$-H\"older continuous super-potentials. Then the measure $T_v \wedge T_h$ is well-defined
and is the Monge-Ampère of a H\"older continuous 
psh 
function. In particular, it has positive Hausdorff dimension. 
\end{proposition}

\begin{proof}
It follows from
Definition \ref{d:holder-sp}
that the intersection $T_v \wedge T_h$ is well-defined.
It is a positive measure since both $T_v$ and $T_h$
are positive.
By \cite{DN},
in order to show that this measure is
a Monge-Ampère as in the statement,
it is enough to show that $u\mapsto\langle T_v\wedge T_h, u\rangle$
is H\"older continuous on $\psh_1 (D)$ with respect to the distance induced by the norm $L^1$. 
By Lemma \ref{l:T-holder-L1} and Definition \ref{d:holder-sp},
for every $u_1, u_2\in \psh_1 (D)$
we have
\[
|\langle T_v\wedge T_h, u_1 - u_2 \rangle|
=|\langle T_h, u_1 T_v - u_2 T_v\rangle|
\lesssim 
\|u_1 T_v - u_2 T_v\|_{-l}^{\alpha}
\lesssim
\|u_1-u_2\|^{\alpha \gamma},
\]
for some positive constant $\gamma$ as in Lemma \ref{l:T-holder-L1}.
The first assertion follows.
The last property is true for
the
Monge-Ampère measure associated with any
H\"older continuous psh function, see for instance
\cite[Théorème 1.7.3]{Sibony99}.
\end{proof}

\section{H\"older regularity of the Green currents}\label{s:holder}

In this section,
we prove the following theorem. It gives the first assertion
and,
by Proposition \ref{p:MA-gen}, 
the third assertion of 
Theorem \ref{t:intro-all}.
We fix an integer $1\le p\le k-1$ and convex open bounded subsets   $ M\Subset \mathbb C^p$ and $ N \Subset \mathbb C^{k-p}$ and set $D:= M\times N$.  We also fix  a Hénon-like map $f$ from a vertical open subset of $D$ to a horizontal open subset of $D$, and convex open subsets 
$M'' \Subset M' \Subset M$ and $N'' \Subset N' \Subset N $  satisfying 
\begin{equation}\label{e:hp-inclusion}
f^{-1}(D) \subset M'' \times N
\quad \mbox{ and } \quad
f(D) \subset M \times N''.
\end{equation}

\begin{theorem}\label{t:super-Holder}
Let $f$ be a H\'enon-like map 
as above.
If $d_{p-1}^+<d$, then $T^+$ has $(l,\alpha)$-H\"older continuous super-potentials for every $0<l<\infty$ and some $0<\alpha<1$ depending on $l$.
\end{theorem}

The following lemma is the precise version of what was announced in Remark \ref{r:interpolation-prelim}.

\begin{lemma}\label{l:interpolation:T+}
Let $f$ be a H\'enon-like map as above.
If $T^+$ has  $(l,\alpha)$-H\"older continuous super-potentials for some $0<l<\infty$ and $0<\alpha<1$, then it has $(l',\alpha')$-H\"older continuous super-potentials for every $0<l'<\infty$ and  some $0<\alpha'<1$ depending on $l,l',\alpha$.
\end{lemma}
\begin{proof}
By \eqref{e:hp-inclusion} and the fact that $f^* (T^+)=T^+$,
we are allowed to make small modifications to the sets $M'$ and $M''$ in Definition \ref{d:holder-sp}, and the fact that
$T^+$ has H\"older continuous 
{super-potentials}
(for some $l$ and $\alpha$)
does not depend on the choice
of $M'$ and $M''$ satisfying \eqref{e:hp-inclusion}.
For simplicity, we then
allow ourself to slightly modifying these domains 
in all the inequalities below. {For example, the norm in the third term of
\eqref{e:for-inv-norm} 
below needs to be taken with a choice of $M''$ which is slightly larger than that in the other two terms.}

 By interpolation \cite{Triebel}, for
 every
$0<l<l'<\infty$
and every bounded 
set $\mathcal F\subset \Psh_h (M'\times N')$ for $\|\cdot\|_{*}$ we have
\begin{equation}\label{e:for-inv-norm}
\|\Phi\|_{-l'}
\leq \|\Phi\|_{-l}
\leq c_{l,l',\mathcal F} \|\Phi\|_{-l'}^{l/l'}
\quad \mbox{ for every }
\Phi \in \mathcal F,
\end{equation}
where the constant $0<c_{l,l',\mathcal F}<\infty$ depends only on 
$l$, $l'$, 
and $\mathcal F$. In particular,
$\mathcal F$ is bounded for
$\|\cdot\|_{-l}$
if and only if it is bounded for
$\|\cdot\|_{-l'}$.
Moreover,
we have
\[
\dist_{-l'} (\Phi_1, \Phi_2)
\leq \dist_{-l} (\Phi_1, \Phi_2)
\leq c_{l,l',\mathcal F^\star} \dist_{-l'} (\Phi_1, \Phi_2)^{l/l'}
\quad \mbox{ for every } \Phi_1, \Phi_2 \in \mathcal F.
\]
The assertion follows.
\end{proof}

 Because of Lemma  \ref{l:interpolation:T+}, it is enough 
to prove that $T^+$ has $(2,\alpha)$-H\"older continuous
super-potentials, for some $0<\alpha<1$. 
Morever, we can just say that 
$T^+$
\emph{has H\"older continuous super-potentials}, with not reference to the specific $l$
 and $\alpha$.
This justifies the phrasing of Theorems \ref{t:super-Holder} and \ref{t:intro-all}.

\bigskip

We fix in the following a constant 
$1< \delta < d/ d^{+}_{p-1}$.
Consider $\Phi_1, \Phi_2\in \Psh_h(M'\times N')$
with 
{$\|\Phi_1
-
\Phi_2\|_{-2}\leq 1$}.
Observe that the currents $\ddc\Phi_i$ are positive and their masses are locally bounded, see Lemma \ref{l:l+2} (ii).
Setting $\lambda:= \|\Phi_1-\Phi_2\|_{-2}$, we need to show that
\begin{equation}\label{eq:holder-goal}
|\langle T^+, \Phi_1-\Phi_2\rangle| \lesssim \lambda^\alpha,\end{equation}
for some $\alpha>0$ and some
implicit constant both independent of $\Phi_1$ and $\Phi_2$.

\medskip

We can assume $\lambda\ll 1$.
By a similar argument as in
Lemma
\ref{l:l+2} (i),
we have
\begin{equation}\label{e:ddcPhilam}
\|\ddc\Phi_1-\ddc\Phi_2\|_{-4} \lesssim \lambda,
\end{equation}
where the implicit constant is independent of $\Phi_1, \Phi_2$.

\medskip

For every $n\in \mathbb N$ and $i \in \{1,2\}$, 
define $\Xi_{i,n}:=d^{-n} (f^n)_*(\ddc\Phi_i)$. These currents are well-defined and horizontal on $M\times N''$.
Since $d^+_{p-1}<d$, by the definition of $d^{+}_{p-1}$ and 
the choice of $\delta$
we have
\begin{equation}
\label{eq:mass-Xi}
\|\Xi_{i,n}\|_{M\times N''}\lesssim \delta^{-n}.
\end{equation}
Observe that the quantity 
in the 
left-hand side of the above
expression
is
well-defined
since $\Xi_{i,n}\geq 0$.

\begin{lemma}\label{l:XiA}
There exists two constants $A>1$ and $C>0$ independent of $\Phi_1$ and $\Phi_2$
such that
\[
\|\Xi_{1,n} -  \Xi_{2,n}\|_{-4}
\leq C A^{n}\lambda
\quad
\mbox{ for every } n\in \mathbb N\]
\end{lemma}

Observe that, as $\Xi_{i,n}$ is horizontal on $M\times N''$, it is also horizontal on $M'\times N'$. Hence the quantity in the statement
is well-defined.

\begin{proof}
For every $n\in \mathbb N$ and
every smooth $(p-1,p-1)$-form
$\Omega$ with vertical support in $M'\times N$,
using \eqref{e:ddcPhilam} 
we have
\begin{align*}
    |\langle\Xi_{1,n} -\Xi_{2,n}, \Omega\rangle|
    & =    
    |\langle d^{-n} (f^n)_*(\ddc\Phi_1)-d^{-n} (f^n)_*(\ddc\Phi_2),\Omega\rangle|\\
    &= d^{-n}|\langle \ddc\Phi_1-\ddc\Phi_2, (f^n)^*(\Omega)\rangle| \\
    &\leq \|\ddc\Phi_1-\ddc\Phi_2\|_{-4}
    \cdot \|(f^n)^*(\Omega)\|_{\Cc^{4} }\\
    &\lesssim  \|(f^n)^*(\Omega)\|_{\Cc^{4} } \, \lam,
\end{align*}
where the implicit constant in the last inequality
is independent of $n,\Phi_1, \Phi_2$, and $\Omega$.
Since we have
\begin{equation}\label{e:pullbackC4}
\|(f^n)^*(\Omega)\|_{\Cc^{4}}\leq 
A^n \|\Omega\|_{\mathcal C^{4}}\end{equation}
for some positive constant $A$ 
independent of
$n$ and $\Omega$, the assertion follows.
\end{proof}

For $i\in \{1,2\}$,
let $\Phi_{i,n}$ be the negative solution of $\ddc \Phi_{i,n}=\Xi_{i,n}$ given by 
Theorem \ref{l:ddcpsi}. These currents are well-defined and horizontal on $M'\times N'$.
Define 
$$\Phi^\star:=\Phi_1-\Phi_2 \quad \text{and} \quad \Phi_n^\star:=\Phi_{1,n}-\Phi_{2,n}$$
and observe that $\Phi^\star\in \Dsh_h (M'\times N')$
and
$\Phi^\star_n\in \Dsh_h (M'\times N')$
for all $n\in \mathbb N$.

\begin{lemma}\label{l:PhiA}
 For every $i\in \{1,2\}$ and
$n\in \mathbb N$,
we have

\[
\|\Phi_n^\star\|_{-3}
=\|
\Phi_{1,n}-\Phi_{2,n}\|_{-3} \lesssim A^n\lambda
\quad
\mbox{ and } 
\quad
 \|\Phi_{i,n}\|_{M'\times N'}\lesssim \delta^{-n},
 \]
where $A>1$ is as in Lemma \ref{l:XiA}
and the implicit constants are independent of $\Phi_1$, $\Phi_2$, and $n$.
\end{lemma}

\begin{proof}
The first inequality follows from Theorem \ref{l:ddcpsi}
and Lemma \ref{l:XiA}.
The second inequality follows from Theorem
\ref{l:ddcpsi} and \eqref{eq:mass-Xi}.
\end{proof}

Define also
$$\Psi_0:=\Phi^\star-\Phi_0^\star \quad \text{and} \quad \Psi_n:=d^{-1}f_*(\Phi_{n-1}^\star)-\Phi_n^\star.$$
Observe that 
$\ddc\Psi_n=0$
for every $n\in\mathbb N$
and that, 
by 
the first inequality in Lemma \ref{l:PhiA},
we have
  \begin{equation}
  \label{eq:Psi-n-l}
  \|\Psi_n\|_{-3}\lesssim A^n\lambda.
  \end{equation}
Observe also that 
both $\Phi^\star_n$ and $\Psi_n$
are differences of positive currents, whose mass is bounded by (a constant times) $\delta^{-n}$ by the second inequality in 
Lemma \ref{l:PhiA}.

\begin{lemma}\label{l:telescope}
For every $n\in \mathbb N$, we have
\begin{equation}\label{e:telescope}
d^{-n} (f^n)_* (\Phi^\star) = \sum_{m=0}^n d^{-n+m} (f^{n-m})_* (\Psi_m) + \Phi^\star_n.
\end{equation}
\end{lemma}

\begin{proof}
By the definition of $\Psi_m$,
for every $n\in \mathbb N$
the right hand side of \eqref{e:telescope}
is equal to
\[
\begin{aligned}
& 
d^{-n} (f^n)_* (\Phi^\star - \Phi^\star_0)
+\sum_{m=1}^n d^{-n+m} (f^{n-m})_* 
\big(d^{-1} f_* (\Phi^\star_{m-1}) - \Phi^\star_m  \big)
+ \Phi^\star_n\\
 & =
 d^{-n} (f^n)_* (\Phi^\star 
 -\Phi^\star_0)
 +
\sum_{m=1}^n d^{-n+m {  -} 1} (f^{n-m {  +} 1})_* 
(\Phi^\star_{m-1})
-
\sum_{m=1}^n d^{-n+m} (f^{n-m})_* 
(\Phi^\star_{m})
+ \Phi^\star_n\\
 & =  
 d^{-n} (f^n)_* (\Phi^\star) 
 -
 d^{-n} (f^n)_* (\Phi^\star_0)
+
d^{-n} (f^n)_* (\Phi^\star_0)
- \Phi^\star_n  + \Phi^\star_n \\
& =
d^{-n} (f^n)_* (\Phi^\star).
\end{aligned}\]
The assertion follows.
\end{proof}

Choose a smooth positive  $(p,p)$-form $\Theta$
with vertical support in $M''\times N$ and of slice mass $\|\Theta\|_v=1$.
We have that $d^{-n} (f^n)^*(\Theta)$ converges to $T^+$. Equivalently, we have
$$T^+=\Theta+\sum_{n\geq 0} d^{-n} (f^n)^* (\Theta'), \quad \text{where} \quad \Theta':=d^{-1} f^*(\Theta)-\Theta.$$
Hence, 
$$\langle T^+,\Phi^\star\rangle = \langle \Theta,\Phi^\star\rangle + \sum_{n\geq 0} \langle d^{-n} (f^n)^* (\Theta'),\Phi^\star\rangle.$$
By Lemma  \ref{l:telescope}, for every $n\in \mathbb N$ we have
$$\langle d^{-n} (f^n)^* (\Theta'),\Phi^\star\rangle = \sum_{m=0}^{n} \langle d^{-n+m} (f^{n-m})^*(\Theta'),\Psi_m\rangle + \langle \Theta', \Phi_n^\star\rangle.$$
It follows that
\begin{equation}\label{e:est-sum-before-cut}
\begin{aligned}
\langle T^+,\Phi^\star\rangle 
&
= \langle \Theta,\Phi^\star\rangle + \sum_{n\geq 0} \langle \Theta', \Phi_n^\star\rangle + \sum_{m\geq 0} \sum_{n\geq m} \langle  d^{-n+m} (f^{n-m})^*(\Theta'), \Psi_m\rangle
\\
& = \langle \Theta,\Phi^\star\rangle + \sum_{n\geq 0} \langle \Theta', \Phi_n^\star\rangle + \sum_{m\geq 0} \langle T^+-\Theta, \Psi_m\rangle.
\end{aligned}\end{equation}

Theorem \ref{t:super-Holder}
will follow from the next two lemmas, which also guarantee that the above series converge absolutely.

\begin{lemma}\label{l:est-Theta'}
There exists a positive constant $C_1$ such that
for every $n\in \mathbb N$ we have
\[
|\langle\Theta', \Phi_n^\star\rangle|
\leq C_1
\min (A^n\lambda, \delta^{-n}).
\]
\end{lemma}

\begin{proof}
Recall that
$\Theta'=d^{-1}f^* (\Theta)-\Theta$,
where both $\Theta$ and $d^{-1}f^* (\Theta)$
are positive, 
vertical,
and of slice mass $1$. 
Recall also
that, for every $n\in \mathbb N$,
$\Phi^\star_n$ is the difference of
two positive currents 
whose mass is less than or equal to a constant times $\delta^{-n}$. This shows the inequality
$|\langle\Theta', \Phi_n^\star\rangle|
\leq C_1
\delta^{-n}$ for some positive constant $C_1$ independent of $n$.

On the other hand, since $\Theta$ is smooth, the same is true for $\Theta'$ and we have $\|\Theta'\|_{  3}\lesssim 1$.
By 
Lemma \ref{l:PhiA},
we deduce that 
\[|\langle\Theta',\Phi^\star_n\rangle|\leq 
 \|\Phi^\star_n\|_{  -3} \cdot \|\Theta'\|_{  3}\lesssim A^m \lambda.\]
 The assertion follows by possibly increasing $C_1$.
\end{proof}

\begin{lemma}\label{l:est-T+-Theta}
There exists a positive constant $C_2$ such that
for every $m\in \mathbb N$ we have
\[
|\langle T^+ - \Theta, \Psi_m\rangle|
\leq C_2
\min (A^m\lambda, \delta^{-m}).
\]
\end{lemma}

\begin{proof}
Recall that $dd^c \Psi_m=0$ for all $m\in \mathbb N$, and that $\Psi_m$ is the difference of two positive currents whose mass is smaller than a constant times $\delta^{-m}$. 
By \eqref{eq:Psi-n-l},
we also have
$\|\Psi_m\|_{  -3}\lesssim A^m \lambda$. As a consequence, both the sequences $\{\delta^{m} \Psi_m\}_{m\in \mathbb N}$
and $\{A^{-m}\lambda^{-1} \Psi_m\}_{m\in \mathbb N}$
belong to a compact subset
of the space of $dd^c$-closed horizontal currents of
bidimension $(p,p)$. As the current $T^+ - \Theta$ is independent of $m$ 
{ and $T^+$ has continuous super-potentials \cite{DNS}}, the assertion follows.
\end{proof}

\begin{proof}[End of the proof of Theorem \ref{t:super-Holder}]
We continue
to use the notations introduced above. Recall that, by Lemma 
\ref{l:interpolation:T+}, we only need to show 
that $T^+$ has $(2,\alpha)$-H\"older continuous
superpotential, i.e., we need to prove \eqref{eq:holder-goal}.

\medskip

It follows from \eqref{e:est-sum-before-cut} and Lemmas \ref{l:est-Theta'} and \ref{l:est-T+-Theta}
 that, 
 for any $N\in \mathbb N$, 
we have
\begin{eqnarray*}
|\langle T^+,\Phi^{\star}\rangle| &\leq&  |\langle \Theta,\Phi^{\star}\rangle| + \sum_{n\leq N} |\langle \Theta', \Phi_n^\star\rangle| + \sum_{n>N} |\langle \Theta', \Phi_n^\star\rangle| + \\
&& +\sum_{m\leq N} |\langle T^+-\Theta, \Psi_m\rangle|+  \sum_{m>N} |\langle T^+-\Theta, \Psi_m\rangle| \\
&\lesssim& \lambda +  \sum_{n\leq N} A^n\lambda + \sum_{n>N} \delta^{-n} + \sum_{n\leq N} A^n\lambda + \sum_{n>N} \delta^{-n}\\
& \lesssim & A^N\lambda +\delta^{-N},
\end{eqnarray*}
where the implicit constants are independent of $N$. 
Choosing 
\[N:= 
\left \lfloor
\frac{ |\log \lambda|}{\log (A\delta)}
\right\rfloor,\] 
we obtain
$$|\langle T^+,\Phi\rangle| \lesssim e^{\log A \cdot |\log \lambda| / \log (A\delta)} \lambda
=
\lambda^{\log \delta / \log (A\delta)}.$$
Recalling that $\delta>1$ is any constant smaller than $d/d^+_{p-1}>1$, this shows that
\[|\langle T^+,\Phi\rangle| \lesssim \lambda^{\alpha}
\quad \mbox{ for every } \quad 
\alpha<\frac{\log d - \log d^+_{p-1}}{\log d - \log d^+_{p-1} + \log A} .\] The assertion follows.
\end{proof}

\endproof

\section{Laminarity of the Green currents}

In this section we complete the proof of
Theorem \ref{t:intro-all}.
We first present some preliminary results on woven currents and shadows of currents in Sections \ref{ss:woven} and \ref{ss:shadows}, respectively. 
These first two sections do not have dynamical content.
We then complete the proof of Theorem \ref{t:intro-all} in Sections \ref{ss:proof-up-to-prop} and \ref{ss:proof-prop-In}.

\subsection{Woven currents}
\label{ss:woven}

Let $D\subset \mathbb{C}^k$ be a bounded convex open set and $\omega$ be the standard K\"{a}hler form of $\mathbb{C}^k$.  For 
every integer $1\le {q}\le k-1$, 
we denote by
$\mathcal{B}_{q}$
the set of 
{pure}
${q}$-dimensional,
not necessarily 
{connected nor} closed,
complex manifolds
$Z\subset D$ 
such that for any compact subset $K\subset D$
we have
$$\int_Z \omega^{q}|_K<\infty.$$
Thanks to Wirtinger's theorem, 
the last condition means that the $2{q}$-dimensional volume of $Z$ (counted with multiplicity) is locally finite. Hence, the current of integration $[Z]$ is a well-defined current  of bidimension $({q},{q})$ on $D$, and in particular we have
\[
\langle[Z],\omega^{q}|_K\rangle
=\int_Z \omega^{q}|_K <\infty
\] 
for every 
compact subset $K\subset D$.
We 
say that a
measure $\nu$ on $\mathcal{B}_{q}$ is $\mathcal{B}_{q}$-\textit{finite} if  
$$\int_{Z\in\mathcal{B}_{q}}\langle[Z],\omega^{q}|_K\rangle d\nu(Z)<\infty $$
for any  compact subset $K\subset D$. 

\begin{definition}
[\cite{BLS93,Dinh05suites}]
\label{d:woven}
    Let $S$ be a positive closed current of bidimension $({q},{q})$ on $D$. We say that $S$ is \textit{woven} if there exists a 
    $\mathcal{B}_{q}$-finite measure $\nu$ on $\mathcal{B}_{q}$ such that 
    \begin{equation}\label{eq:woven}
        S=\int_{Z\in\mathcal{B}_{q}}[Z]d\nu(Z).
    \end{equation}
    We say that $S$ is \textit{laminar} if, in addition to \eqref{eq:woven}, for  $(\nu\otimes \nu)$-almost  all $(Z,Z') \in \mathcal{B}_{q}^2$
    the set $Z\cap Z'$  is
    {open
    (possibly empty)
    in $Z$ and $Z'$.} 
   \end{definition}

In order to show that the Green currents are woven, we will
use the a criterion due to 
 de Thélin  \cite{dTh04}
 (see also {\cite{Dinh05suites,Dujardin03}} for a {global} version),
 that we now recall.
Let $\mathbb{G}({q},k)$ be the Grassmannian parametrizing
the 
 linear subspaces of dimension $q$ of
$\mathbb C^k$.
 Recall that $\mathbb G ({q},k)$ is a 
projective variety, 
of dimension $q(k-q)$,
and 
 that
it can naturally be seen as a submanifold of 
 the projective space
$\mathbb P (\bigwedge^q \C^k) \sim \mathbb P^{L({q},k)}$
 for some integer $L(q,k)$, where  
 the tangent space at each point of $\C^k$ is identified to $\C^k$.
We will denote by
$\omega_{\G}$ the K{\"a}hler form on $\G ({q},k)$ induced by the natural Fubini-Study form on $\P^{L({q},k)}$. 
We will let $\pi_D$ and $\pi_\mathbb{G}$
(resp. $\pi_D$ and $\pi_{\mathbb P^{L}}$)
be the natural projections of $D\times  \mathbb{G}({q},k)$
(resp. $D\times  \mathbb{P}^{L({q},k)}$)
to its first and the second factors, respectively.
From now on, to emphasize the factor
 $D$, 
we will use the notation $\omega_D$ instead of $\omega$
to denote
the standard
K\"{a}hler form on $D\subset \mathbb{C}^k$. 
Observe that
$(\pi_{\mathbb{G}})^*\omega_{\G}+ (\pi_D)^*\omega_D$ 
is 
a K\"{a}hler form on $D\times  \mathbb{G}({q},k)$.
 Volumes in $D\times \G ({q},k)$
will be computed with respect to this K{\"a}hler form.

For any smooth
complex ${q}$-dimensional submanifold $\Sigma \subset D$, we define 
 a
${q}$-dimensional submanifold 
$\hat \Sigma$
of $D\times \mathbb{G}({q},k)$ as
\begin{equation}\label{eq:hatsigma}
   \hat{\Sigma}
   :=
   \{(z,H)\in D\times \mathbb{G}({q},k): 
   z\in  \Sigma; H \mbox{ is parallel to } T_z \Sigma\},
\end{equation}
where $T_z\Sigma$ is the tangent space of $\Sigma$ at $z$. We can extend the above definition to the case of $\Sigma$ 
not smooth
by defining $\hat \Sigma$ 
as the closure of $\hat \Sigma_r$, where 
$\Sigma_r$
is the regular part of $\Sigma$. We will call
$\hat \Sigma$ 
the \emph{lift}
of $\Sigma$ to $D\times \mathbb{G}({q},k)$. 
The role of $\hat \Sigma$ in the theory
was observed in \cite{Dinh05suites}.

\begin{theorem}[{de Thélin \cite{dTh}}]\label{th:dTh}
    Let $\Sigma_n\subset D$ be a sequence of submanifolds of pure dimension $1\leq {q}\leq k-1$ and $ \hat{\Sigma}_n$ be the lift of $\Sigma_n$ to $D\times \mathbb{G}({q},k)$. 
 Let $v_n$ and $\hat{v}_n$ be the volumes of $\Sigma_n$ and $\hat{\Sigma}_n$, respectively. Assume  that all
 $v_n$ and $\hat{v}_n$ are
 finite and that the sequence $v_n^{-1}[\Sigma_n]$ converge to a current $T$ on $D$.
 If $\hat{v}_n=O(v_n)$ as $n\to \infty$, 
then the current $T$ is woven.
\end{theorem}

 The following criterion, see for instance
\cite{dThD},
will imply the laminarity of the Green current $T^-$ when $p=k-1$ in Theorem \ref{t:intro-all}.

\begin{proposition}\label{p:crit-lam}
    Let $T$ be a woven positive closed  $(1,1)$-current on $D$. Assume the local potentials of $T$ are integrable with respect to $T$. If  $T\wedge T=0$, then $T$ is laminar.
\end{proposition}

Observe that the integrability of the local potentials of $T$ with respect 
to $T$ guarantees that the intersection
$T\wedge T$ is well-defined.

\subsection{Shadows of currents on $D\times  \mathbb{G}({q},k)$}\label{ss:shadows}

In this section, we consider again a bounded convex open subset $D\subset \mathbb C^k$ and, for a given fixed integer
$1\leq {q}\leq k-1$,
the Grassmannian $\mathbb G ({q},k)$ as in the previous section. We will consider currents on $D\times \mathbb G ({q},k)$ and define a suitable projection of these currents on $D$, that we will call their shadows. This notion is related
to the \emph{h-dimension} as in \cite{DS18density}
(see also \cite{Nguyen21,Vu21})
and allows one to detect the excess of
dimension in
{the vertical} directions.
For simplicity,
we will only consider the product space 
$D\times \mathbb G ({q},k)$ that we will need later, but most of this section generalizes to arbitrary product spaces. 
 In particular, the construction could be carried out on $D\times \mathbb P^{L({q},k)}$.

\begin{definition}
\label{d:D-dimension}
      Let $S\neq 0$ be a
      (non necessarily closed)
      positive
    current of bidimension $({q},{q})$  on $D\times\mathbb{G}({q},k)$. 
Let $l_0$ be such that
 \begin{enumerate}
    \item $S\wedge (\pi_D)^*({\omega_D^{l_0}})\neq 0$;
        \item $S\wedge (\pi_D)^*\omega_D^{l_0+1}=0$.
    \end{enumerate}
We call
$l_0$ the \emph{$D$-dimension}
of $S$.
\end{definition}

Observe that $l_0$ as in the above definition always exists and 
satisfies $l_0\leq q\leq k$.

{\begin{lemma}\label{l:CS}
   Let $S$ and $l_0$ be as in Definition \ref{d:D-dimension}. We have
$S \wedge (\pi_D)^* \alpha =0$
for every smooth $(2l_0+1)$-form $\alpha$ on $D$.  
\end{lemma}
\begin{proof}
     We can assume that $l_0< q$
(as otherwise the statement is clear)
and 
it is enough to consider the case where $\alpha$
is of the form $\alpha = \beta \wedge \Omega$
for some $(\ell,\ell)$-form $\beta$
with $\ell\le l_0$ and 
$(2l_0-2\ell +1,0)$-form $\Omega$, and  show that
$\langle S \wedge (\pi_D)^* \alpha, \Phi\rangle=0$ for every compactly supported
smooth
$(q+\ell-2l_0-1, q-\ell)$-form
$\Phi$ of the form $\Phi= \gamma\wedge \Theta$, for some
$(q+\ell-2l_0-1,q+\ell-2l_0-1)$-form
$\gamma$ and $(0,2l_0-2\ell +1)$-form  $\Theta$.
 
  Applying Cauchy-Schwarz inequality, we have
\[
\begin{aligned}
|\langle S\wedge (\pi_D)^* \alpha, \Phi\rangle|^2
& =
|\langle S \wedge (\pi_D)^* \beta \wedge \gamma,
(\pi_D)^* \Omega \wedge \Theta\rangle|^2\\
& \leq
|\langle S \wedge (\pi_D)^* \beta  \wedge \gamma,
(\pi_D)^* (\Omega \wedge  \overline \Omega)\rangle|
\cdot
|\langle S \wedge (\pi_D)^* \beta  \wedge \gamma,
\Theta\wedge \overline \Theta\rangle|=0.
\end{aligned}\]
Since $(\pi_D)^*(\beta  \wedge 
\Omega \wedge  \overline \Omega)$ is a smooth $(2l_0-\ell+1,2l_0-\ell+1)$-form and $\ell\le l_0$,
the first  factor in the last term vanishes
by the assumption (ii) in Definition \ref{d:D-dimension}.
\end{proof}
}

\begin{definition}\label{d:shadow}
        Let $S$ be a (non necessarily closed)
positive  
current of bidimension $({q},{q})$  on $D\times\mathbb{G}({q},k)$
and 
$l_0$
the $D$-dimension of $S$ as in Definition \ref{d:D-dimension}. 
   The \emph{shadow} of $S$ on $D$
   is the 
   $(l_0,l_0)$-bidimensional
   current  $ \underline{S}$ on $D$
      given by 
   \[
\underline S := (\pi_D)_* (S\wedge 
(\pi_{\G})^*\omega_{\G}^{{q}-l_0} ).   
   \]
\end{definition}

Observe that $\underline S$ satisfies
    \begin{equation}\label{eq:defshadow}
        \langle\underline{S},\theta\rangle
        =\langle {S},(\pi_{\G})^*\omega_{\G}^{{q}-l_0}\wedge (\pi_D)^*\theta\rangle=\int_{D\times \G({q},k)} S\wedge(\pi_{\G})^*\omega_{\G}^{{q}-l_0}\wedge(\pi_D)^*\theta,
    \end{equation}
for every compactly supported smooth $(l,l)$-form
$\theta$ on
  $D$.  
  In particular,
  the shadow $\underline{S}$ of $S$
is well-defined 
since $(\pi_{\G})^*\omega_{\G}^{{q}-l_0}\wedge(\pi_D)^*\theta$ is a compactly supported smooth $({q},{q})$-form on $D\times \G({q},k)$.

\medskip

The following proposition collects some properties of the shadows that we will need in the sequel and justifies the terminology in Definitions 
\ref{d:D-dimension} and \ref{d:shadow}.

\begin{proposition}\label{p:shadow}
Let $S$ be a positive  current of bidimension $({q},{q})$  on $D\times\mathbb{G}({q},k)$
and 
$l_0$ be its $D$-dimension.
  Then the 
  shadow $\underline{S}$ of $S$ 
  is a 
   non zero
  positive  current of bidimension $(l_0,l_0)$ on $D$, satisfying
  \begin{equation}\label{eq:shadowmass}
     \|\underline{S}\|_{{\tilde{D}}}= \|S \wedge (\pi_D)^*({\omega_D^{l_0}})\|_{{\tilde{D}}\times\mathbb{G}({q},k)}
  \end{equation}
  on any open subset ${\tilde{D}} \Subset D$.
 Moreover, if $S$ is closed on {$U\times \mathbb{G}({q},k)$} for some open subset $U\subset D$,  
 then $\underline S$ is closed on $U$.
  \end{proposition}
\begin{proof}  Since $S$ is a positive  current on $D\times \G({q},k)$, by 
\eqref{eq:defshadow} $\underline{S}$ is a  positive current of bidimension $(l_0,l_0)$
on $D$. Since  $S\wedge (\pi_D)^*({\omega_D^{l_0+1}})=0$ we have
\begin{align*}
    \|\underline{S}\|_{{\tilde{D}}}=\int_{{\tilde{D}}} \underline{S}\wedge \omega_D^{l_0}
  &=\int_{{\tilde{D}}\times \G({q},k)} S\wedge (\pi_{\mathbb{G}})^*\omega_{\G}^{{q}-l_0}\wedge (\pi_D)^*\omega_D^{l_0}\\
    &=\int_{{\tilde{D}}\times \G({q},k)} S\wedge (\pi_D)^*\omega_D^{l_0}\wedge \left((\pi_{\mathbb{G}})^*\omega_{\G}+ (\pi_D)^*\omega_D\right)^{{q}-l_0} \\
    &=\|S\wedge (\pi_D)^*\omega_D^{l_0}\|_{{\tilde{D}}\times \G({q},k)}.
\end{align*}
 Hence,
 \eqref{eq:shadowmass} holds. 
 
Let us now prove the last assertion. Since the problem is local, we can assume that $U$ is a small ball compactly supported in $D$. 
Assume first that we have
$l_0=0$. In this case, since $\underline{S}$ is a measure on $D$, it is clear that $\underline{S}$ is closed.  
Assume now that $l_0>0$ and  
let $\theta$ be a smooth {$(l_0-1,l_0)$}-form
compactly supported
on $U$. By using Stokes' formula twice and the fact that $S$ is a closed  current of bidimension $({q},{q})$  on $U\times\mathbb{G}({q},k)$,
we have
\begin{align*}
\langle d\underline{S},\theta\rangle&=\langle \underline{S},d\theta\rangle=\langle \underline{S},\partial\theta\rangle=\langle {S},(\pi_{\mathbb{G}})^*\omega_{\G}^{{q}-l_0}\wedge\pi_D^*(\partial\theta)\rangle=\langle {S},(\pi_{\mathbb{G}})^*\omega_{\G}^{{q}-l_0}\wedge\pi_D^*(d\theta)\rangle\\
&=\langle {S}\wedge (\pi_{\mathbb{G}})^*\omega_{\G}^{{q}-l_0},d(\pi_D^*(\theta))\rangle=\langle d{(S\wedge (\pi_{\mathbb{G}})^*\omega_{\G}^{{q}-l_0})},\pi_D^*(\theta)\rangle\\
&=\langle dS\wedge (\pi_{\mathbb{G}})^*\omega_{\G}^{{q}-l_0},\pi_D^*(\theta)\rangle =0.
\end{align*}
By similar arguments, we can also obtain that 
$\langle d\underline{S},\theta\rangle=0$
 when $\theta$ is a smooth $(l_0,l_0-1)$-form
 compactly supported
 on $U$. Hence, $\underline{S}$ is 
 closed on $U$ and the proof is complete.
  \end{proof}

\subsection{End of the proof of Theorem \ref{t:intro-all}}
\label{ss:proof-up-to-prop}

We use the notations as at the beginning of Section \ref{s:holder}.
We also set
$D' := M' \times N'$.
The following theorem completes the proof of Theorem \ref{t:intro-all}.

\begin{theorem}\label{t:intro-green-laminar}
Let $f$ be a H\'enon-like map on $D$
 as above.
If $d_{p-1}^+<d$,
then $T^{-}$ is woven.
If, furthermore, we have 
$p=k-1$, then $T^-$ is laminar.
\end{theorem}

Let $\Sigma$ be a horizontal $p$-dimensional plane of $M\times N'$. By Theorem \ref{t:greencurrents}
we have $d^{-n}(f^n)_*[\Sigma]\to T^-$.
Let $\hat{\Sigma}$ be the submanifold of $D\times  \mathbb{G}(p,k)$ defined
by means of 
\eqref{eq:hatsigma} and set $R:=[\Sigma]$ and $\hat{R}:=[\hat{\Sigma}]$. As above,  we denote by  $\omega_\G$  the
K\"{a}hler form of $\G(p,k)$ induced by the Fubini-Study form on  $\mathbb P^{L(p,k)}$. 
Again, to emphasize $D$,
we use the notation $\omega_D$ instead of $\omega$
to denote
the standard
K\"{a}hler form of $\mathbb{C}^k$. 
Recall that $\pi_D$ and $\pi_\mathbb{G}$ are the natural projections of $D\times  \mathbb{G}(p,k)$ to its first and the second factors, respectively, and 
that 
we use $(\pi_{\mathbb{G}})^*\omega_{\G}+ (\pi_D)^*\omega_D$ as a K\"{a}hler form on $D\times  \mathbb{G}(p,k)$.

\medskip

Consider the map 
\[F: D_{v,1} \times \mathbb{G}(p, k) \to D_{h,1} \times \mathbb{G}(p, k)
\quad \mbox{ defined as }
\quad 
F = (f, Df).\]
Hence, 
for $(z, H_{P}) \in D \times \mathbb{G}(p, k)$, we have
$F(z,H_P) =(f(z), H_{f(P)})$, where $P$ is a $p$-dimensional submanifold of $D$ passing through $z$, and $H_P$ and $H_{f(P)}$ are 
the elements in $\mathbb{G}(p, k)$ representing parallel planes 
to $T_zP$ and $T_{f(z)}f(P)$, respectively. 
Observe that if $Q$ and $Q'$ are $p$-dimensional submanifolds of $D$ passing through $z$, with $T_z Q=T_z Q'$,
we have $H_Q=H_{Q'}$. 
For any  $z_0 \in D_{v,1}$,
we can identify
$F(z_0, \cdot)$ 
to an automorphism $\gamma_{z_0}$ on 
$\mathbb G(p,k)$.
It is then 
clear that $F$ is invertible from its domain
$D_{v,1}\times \mathbb{G}(p, k)$
to its image $D_{h,1}\times \mathbb{G}(p, k)$. 
 The map $F$ can also be seen as a map
from $D_{v,1} \times \mathbb{P}^{L(p, k)}$ to 
$D_{h,1} \times \mathbb{P}^{L(p, k)}$. In particular, every $F(z_0, \cdot)$ 
can also be seen as an automorphism of 
$\mathbb{P}^{L(p, k)}$.

\medskip

The following is the key step towards Theorem \ref{t:intro-green-laminar}.

 \begin{proposition}\label{p:In=Odn}
 Let $f,F,\Sigma,\hat{R}, M''$, and $N$
 be as above.  Then 
 \[\|(F^n)_*\hat{R}\|_{M''\times N\times\G(p,k)}=O(d^{n})
 \quad \mbox{ as } n\to\infty.\]
 \end{proposition}

The proof of Proposition \ref{p:In=Odn}
is given in the next section.
We now conclude 
the proof of Theorem \ref{t:intro-green-laminar} assuming Proposition \ref{p:In=Odn}.

\begin{proof}[{Proof of Theorem \ref{t:intro-green-laminar}}]
 Let $\Sigma$ be a horizontal submanifold of $M\times N'$ of 
  pure dimension $p$ and
  such that  $d^{-n}(f^n)_*[\Sigma]$ converge to $T^-$.
   Let $v_n$ and $\hat{v}_n$
  be the volume of $\Sigma_n:=f^{n}(\Sigma)$  on $M''\times N$ and
  of $\hat{\Sigma}_n$
  on $M''\times N \times\G(p,k)$, respectively, where $\hat{\Sigma}_n$ is the submanifold of $D\times  \mathbb{G}(p,k)$ defined by means of \eqref{eq:hatsigma}. Note that, since $d^{-n}(f^n)_*[\Sigma]$ converge to $T^-$ we have
  \begin{equation}\label{eq:vn=dn}
     d^n\lesssim v_n\lesssim d^n. 
  \end{equation}
   By definition of $\hat{\Sigma}_n$ and $F$ we can see that $\hat{\Sigma}_n=F^n(\hat{\Sigma})$. Hence, we have
   \[\hat v_n=\|(F^n)_*[\hat{\Sigma}]\|_{M''\times N\times\G(p,k)}.\]
   We deduce from Proposition \ref{p:In=Odn} 
   that $\hat{v}_n=O(d^n)$
   as $n\to \infty$.
   It follows from \eqref{eq:vn=dn} that we have $\hat{v}_n=O(d^n)=O(v_n)$.  Hence, thanks to  Theorem \ref{th:dTh}, we conclude that $(T^-)_{|M''\times N}$ is woven on $D$ (or, equivalently, that $T^-$ is woven as a current on $M''\times N$). In order to conclude the proof of the first assertion, we need to deduce that $T^-$ is woven on all of $D$.

 \medskip

Recall that 
$\mathcal{B}_p$ denotes
the set of  $p$-dimensional connected complex submanifolds of $D$ with
 locally finite
 $2p$-dimensional volume. Denote similarly by 
 $\mathcal{B}''_p$
the set of  $p$-dimensional connected complex submanifolds of $M''\times N$ with
 locally finite
 $2p$-dimensional volume.
 As $(T^-)_{|M''\times N}$ is woven, we can write it as 
     \begin{equation}\label{eq:woven-second}
        (T^-)_{|M''\times N}=\int_{Z\in\mathcal{B}''_p }[Z]d\nu(Z),
    \end{equation}
 where  $\nu$ is a
 $\mathcal{B}''_p$-finite measure on $\mathcal{B}''_p$. Define  a map
 \[\mathcal{F}:\mathcal{B}''_p \to \mathcal{B}_p
 \quad
 \mbox{ as }\quad
 Z\mapsto f(Z\cap D_{v,1}).\]
  The map $\mathcal F$ is well-defined by \eqref{e:hp-inclusion}.
   Let us also define a measure $\tilde{\nu}$ on $\mathcal{B}_p$ as
  \[
\tilde \nu := d^{-1} \mathcal F_* \nu.
\]
Since $\nu$ is
$\mathcal{B}''_p$-finite on $\mathcal {B}''_p$, 
$\tilde{\nu}$ is 
$\mathcal{B}_p$-finite on $\mathcal{B}_p$. 
Indeed, let $K$ be any compact subset of $D$. Thanks to \eqref{e:hp-inclusion} there exists a 
 convex open set $M^\star\Subset M''$ such that $f^{-1}(K)\subset M^\star\times N$.  Then, we have
\begin{align*}
  \int_{Z\in\mathcal{B}_{p}}\langle[Z],\omega^{p}|_K\rangle d\tilde{\nu}(Z)  &= d^{-1}\int_{Z\in{\mathcal{B}_{p}''}}\langle[f(Z)],\omega^{p}|_K\rangle d{\nu}(Z)\\
  &=d^{-1}\int_{Z\in{\mathcal{B}_{ p}''}}\langle[Z],
  f^*(\omega^{p}|_K)\rangle d{\nu}(Z) \\
  &\lesssim \int_{Z\in{\mathcal{B}_p''}}\langle[Z],\omega^{p}|_{M^\star\times N'}\rangle d{\nu}(Z),
  \end{align*}
 where in the last inequality we used  \cite[Lemma 3.3]{BDR} and the fact that 
 {we have}
 $Z\subset M''\times N'$ {for $\nu$-almost every $Z$}. Since $\nu$ is $\mathcal{B}''_p$-finite,
 we deduce from the last inequality that $\tilde{\nu}$ is 
$\mathcal{B}_p$-finite on $\mathcal{B}_p$.

Again by \eqref{e:hp-inclusion},
   we also have  $d\cdot T^-=f_*\left((T^-)_{|M''\times N}\right)$ on $D$. 
   Hence, by \eqref{eq:woven-second},
   we have
      \[   T^-
             =d^{-1}f_*\left((T^-)_{|M''\times N}\right)
        =d^{-1}\int_{Z\in\mathcal{B}''_p}f_*[Z]d\nu(Z)
        =\int_{Z\in \mathcal{B}_p}[Z]d\tilde{\nu}(Z)
       \]
   on $D$.
 Hence, $T^-$ is woven on $D$.
     \medskip

In order to conclude the 
proof of Theorem \ref{t:intro-green-laminar}, we need to show that
$T^-$ is laminar
assuming that 
$p=k-1$
{(observe that this implies 
$d^{-}_{k-p-1}=d^-_0=1<d$)}. 
Recall that the condition $p=k-1$
 implies that
$T^-$
 is of bidegree $(1,1)$.
 By Theorem \ref{t:super-Holder},
 $T^-$
 has H\"older continuous local potentials.
 So $T^-\wedge T^-$ is a well-defined  horizontal positive closed $(2,2)$-current\footnote{It is actually enough for this 
 to apply \cite[Theorem 4.1]{DNS}, which gives that 
 $T^-$
 has bounded local potentials.}.
 By Proposition \ref{p:crit-lam}, it is enough to show that
 $T^-\wedge T^-\ne 0$.

 Assume by contradiction that $T^-\wedge T^-\ne 0$. By the
 $f$-invariance of $T^-$
 and the fact that $f$ is invertible,
 we have $f_* (T^-\wedge T^-)=d^2(T^-\wedge T^-)$. By Definition \ref{d:degrees-currents}, this
 implies that $d^2\le d^+_{k-2}$. Since by assumption
 we have
 $d^+_{k-2}<d=d^+_{k-1}$, this gives $d^2<d$, which implies that $d<1$.
Since we always have $d\geq 1$, 
 this gives the desired contradiction and shows that $T^-\wedge T^-=0$. The proof is complete. 
\end{proof}

We conclude this section with the  following technical
lemma, that we will need in the proof of 
Proposition \ref{p:In=Odn}. 

\begin{lemma}
    \label{l:subsequence}
For every $0\le \tilde l\le p$ let $\{C_{\tilde l}(n)\}_{n\in\mathbb N}$  be a sequence of positive numbers such that $\{C_{p}(n)\}_{n\in\mathbb N}$  is bounded and $$\limsup_{n\to\infty}\max_{0\le \tilde l\le p-1} C_{\tilde l}(n)=\infty. $$ Then there exists 
a positive constant $\beta$,
an index $0\leq l<p$, and a sequence $\{n_j\}_{j\in\mathbb N}$ such that
\begin{itemize}
\item[{\rm (i)}]  
$C_{\tilde{l}}(n)\leq \beta C_l(n_j)$ 
for all $n\le n_j$ and 
all $0\le\tilde{l}\le p$;
  \item[{\rm (ii)}]   
   we have 
  \[\lim\limits_{j\to\infty}
  \frac{
  C_{l'}(n_j-s)}{C_l(n_j)}=0 
  \quad \mbox{ for all } l<l'\le p
  \mbox{ and all } s\in \mathbb N.\]
\end{itemize}
\end{lemma}

\begin{proof}
Set
$$L:=\{\tilde l:\limsup_{n\to \infty} C_{\tilde l}(n)=\infty\}=:\{l'_1,l'_2,...,l'_{m'}\},$$
where $1\le m'\le p$ and 
$l'_1< \dots <l'_{m'}$.
By assumption, we have $p\notin L$ and so $l'_{m'}<p$. Set
 $$a_n:=\max\{C_{\tilde l}(n): {\tilde l}\in L\}.$$
Since by assumption we have 
$\limsup_{n\to \infty} a_n=\infty$,
there exists a subsequence $a_{\tilde n_j}$ such that $a_n< a_{\tilde n_j}$ for every $j$ and $n< {\tilde n_j}$.  
 We will only use this sequence $\{\tilde n_j\}$ or a subsequence of it.
For every $j$,
define 
$$L_j:=\{\tilde l\in L: C_{\tilde l}({\tilde n_j})=a_{{ \tilde n_j}}\}$$
As $L$ is finite, up to replacing the sequence
$\{{ \tilde n_j}\}$
by a subsequence, we can assume that
$L_j$ does not depend on $j$, i.e., that we can write
\[
\mathcal L := L_j =:  \{l_1, \dots, l_m\}
\quad \mbox{ for every } j.
\]
It is clear that the pair $(l_1,\{{\tilde n_j}\}_{j\in\mathbb N})$ satisfies  
(i)  for $\beta=1$.
We now describe a procedure that,
by possibly modifying the pair and $\beta$, will also lead to (ii).

\medskip

We will play
the following simple game 
between $l_1$ and the other indexes $l_r$, for $1<r\leq m$. 
If there exists 
$s\in \mathbb N$ such that 
\begin{equation}\label{eq:nj-s}
 \limsup\limits_{j\to \infty} \frac{C_{{l}_r}( \tilde n_j-s)}{C_{{l_1}}( \tilde n_j)}>0  \end{equation}
then $l_r$ is the winner. Otherwise, $l_1$ is the winner. 

\medskip

We start playing the game 
above
between $l_1$ and $l_m$. 
If the winner is $l_m$,
we stop playing. Otherwise, we play the game between $l_1$ and $l_{m-1}$. Again, if the winner is $l_{m-1}$ we stop playing, otherwise we continue the game between $l_1$ and $l_{m-2}$. 
After repeating the procedure 
at most $m-1$ times, we have the final winner. If the final winner is $l_1$ then it is clear that the pair $(l_1,\{{\tilde n_j}\}_{j=1}^\infty)$ satisfies  both 
(i)
and (ii).
In this case, the proof is complete. 
We can then assume that the final winner is $l_{r_0}$ for some $1<r_0\le m$.
Observe that this means that $l_{r_0}$ was the winner between $l_1$ and $l_{r_0}$ (i.e., \eqref{eq:nj-s} holds with $r_0$ instead of $r$, for some given choice $s_0$ of $s$),
but also
that $l_1$ was the winner against all $l_{r'}$, for all $r'> r_0$.

\medskip

After again choosing a subsequence, if necessary, we can assume that the limsup
in \eqref{eq:nj-s} 
(with $r=r_0$ and $s=s_0$)
is actually a limit, i.e.,
that
we have
\begin{equation}\label{eq:nj-s-bis}
 \lim \limits_{j\to \infty} \frac{C_{{l}_{r_0}}({\tilde n_j}-s_0)}{C_{{l_1}}({\tilde n_j})}
 =: A >0.
 \end{equation}
We can 
also assume that $\tilde{n}_1>s_0$.
Take {$(l, \{n_j\}_{j\in \mathbb N})$}
as a new pair, 
where {${n}_j:={\tilde n_j}-s_0$ and $l:=l_{r_0}$}.
We now 
show that this pair satisfies 
both the requests 
(i)
and (ii).

\medskip
By \eqref{eq:nj-s-bis},
we have
$C_{l_1}({\tilde n_j})\le \beta\cdot C_{l}(n_j)$ for every $j$, 
for some positive 
constant {$\beta$} independent of $j$. 
By construction,
for any $n\le {\tilde n_j}$ and $0\le \tilde{l}\le p$,
we also 
have $C_{\tilde{l}}(n)\le C_{l_1}({\tilde n_j})$. As a consequence, 
we  have 
   $$C_{\tilde{l}}(n) \le C_{l_1}({\tilde n_j}) \le {\beta}\cdot C_{l}({{n}_j})
   \quad
   \mbox{ for all } 
   n\le n_j
   \mbox{ and }0\le {\tilde{l}}\le p.$$
So, $(l,\{{n}_j\}_{j\in \mathbb N})$
satisfies
(i).
For the sake of contradiction, assume that
$(l,{n}_j)$
does not satisfy (ii).
In this case, 
there must
exist $l_{r_1}$
with $r_0<r_1\le m$ 
and $s_1$ 
such that  
\begin{equation}\label{eq:connj-s}
 \limsup\limits_{j\to \infty} \frac{C_{{l}_{r_1}}({n}_j-s_1)}{C_{l}({n}_j)}>0.
\end{equation}
Again, after choosing a subsequence
if necessary, we can assume that the above $\limsup$ is a limit. Since $l_1$ was 
the winner when playing the game between $l_1$ and $l_{r_1}$,
we have
 $$0=\lim\limits_{j\to \infty} \frac{C_{{l}_{r_1}}({\tilde{n}_j}-s_0-s_1)}{C_{{l_1}}({\tilde n_j})}
 =
 \lim\limits_{j\to \infty}\frac{C_{l}({\tilde n_j}-s_0)}{C_{{l_1}}({\tilde n_j})}\cdot \frac{C_{{l}_{r_1}}({n_j}-s_1)}{C_{{l}}({n_j})}>0. $$  
The last inequality follows from  \eqref{eq:nj-s-bis}
and \eqref{eq:connj-s} and the assumption that the $\limsup$ in \eqref{eq:connj-s} is actually
a limit.
This gives a contradiction. Hence, the pair
$(l,\{{n}_j\}_{j\in \mathbb N})$
satisfies (ii)
and the proof is complete.
\end{proof}

\subsection{Proof of Proposition \ref{p:In=Odn}}\label{ss:proof-prop-In}

Take a smooth vertical cut-off function $0\le\chi\le1 $ on $D$, equal to $1$ in a neighbourhood of $\overline{M}''\times N$  and supported on $M'\times{N}$. 
For every $0\leq l\leq p$
define
$$I_l(n):=\int_{D\times\G(p,k)} (\pi_D)^*(\chi) \cdot (F^n)_*(\hat{R})\wedge (\pi_D)^*\omega_D^l\wedge (\pi_{\mathbb G})^*(\omega_{\mathbb G})^{p-l}$$
and set
$$I(n):=\int_{D\times\G(p,k)} (\pi_D)^*(\chi) \cdot (F^n)_*(\hat{R})\wedge ((\pi_D)^*(\omega_D)+(\pi_{\mathbb{G}})^*\omega_\G)^p=\sum_{l=0}^p {p \choose l} I_l(n).$$ 
Observe
that $\|(F^n)_*(\hat{R})\|_{M''\times N\times\G(p,k)}\le I(n) $. So it is enough to show that $I(n)=O(d^n)$.

\begin{lemma}\label{l:Ip(n)}
We have
\begin{equation}\label{eq:Ip(n)}
  d^n\lesssim I_p(n)\lesssim d^{n}.  
\end{equation}
\end{lemma}

\begin{proof}
By the definitions of $\hat R$ and $F$, we have
\begin{equation}\label{eq:piDF=f}
    (\pi_D)_*( (F^n)_*(\hat{R}))= (f^n)_*(R).
\end{equation}
Indeed, if $\theta$ is a smooth $(p,p)$-form compactly supported on $D$, 
we have
\begin{align*}
    \langle (\pi_D)_*( (F^n)_*(\hat{R})), \theta \rangle_D&=\langle  (F^n)_*(\hat{R}), (\pi_D)^*(\theta) \rangle_{D \times \mathbb{G}(p, k)}\\
    &=\int_{F^n(\hat \Sigma)}(\pi_D)^*(\theta)=\int_{f^n( \Sigma)}\theta=\langle  (f^n)_*({R})), \theta \rangle_D,
\end{align*}
which proves \eqref{eq:piDF=f}.
   By using this identity,
    we have 
\begin{align*}
  I_p(n)  & =\int_{D\times\G(p,k)} (\pi_D)^*(\chi) \cdot (F^n)_*(\hat{R})\wedge (\pi_D)^*\omega_D^p\\
  & =\int_D (\pi_D)_*\left((\pi_D)^*\chi  \cdot (F^n)_*(\hat{R})\wedge (\pi_D)^*\omega_D^p\right)\\
  & =\int_D \chi \cdot (f^n)_*(R)\wedge \omega_D^p\le \int_{M'\times N} (f^n)_*(R)\wedge \omega_D^p\lesssim d^n.
 \end{align*}
The last equality is due to \eqref{eq:dpn=d}. Similarly, since the dynamical degrees do
not depend on the choice of $M'$ and $N'$, by using \eqref{eq:dpn=d} again, we also
have
\begin{align*}
  I_p(n)  &  =\int_D \chi \cdot (f^n)_*(R)\wedge \omega_D^p\ge \int_{M''\times N} (f^n)_*(R)\wedge \omega_D^p\gtrsim d^n.
 \end{align*}
This completes the proof.
 \end{proof}

Recall that, in order to prove
Proposition \ref{p:In=Odn},
 it is enough to show that $I(n)=O(d^n)$.  Assume, for the sake of contradiction, that 
 this is false. For  $0\le \tilde l\le p$, denote 
 \begin{equation}\label{e:CI}
 C_{\tilde l}(n):=
 \frac{I_{\tilde l}(n)}{d^{n}}.
 \end{equation}
 Thanks to Lemma \ref{l:Ip(n)}
 and the contradiction assumption, 
 we can see that the sequences $C_{\tilde l}(n)$ satisfy the conditions of Lemma \ref{l:subsequence}.  Let $l$ and  $\{n_j\}_{j\in\mathbb N}$ be as given by
 Lemma \ref{l:subsequence} applied to these sequences. 
 
 \medskip

For all integers $s$ and 
$j$ such $n_j > s$,
set 
$$S_{j}^{(s)}
:=\frac{(\pi_D)^*\chi \cdot (F^{n_j-s})_*(\hat{R})}{C_l(n_j)d^{n_j-s}}.$$
Each $S_j^{(s)}$ is a current on $D\times \mathbb G(p,k)$.
Since $f$ is a horizontal-like map, we have $(f^s)_* (\chi)\ge \chi$ on  $D_{h,s}$. Since $D_{h,n_j}\subset D_{h,s}$ and $(F^s)_*S_{j}^{(s)}$ is supported  in $D_{h,n_j}\times \G(p,k)$, we have
\begin{equation}\label{eq:fsgedss}
    (F^s)_*S_{j}^{(s)}=\frac{(\pi_D)^*((f^s)_*\chi)\cdot (F^{n_j})_*(\hat{R})}{C_l(n_j)d^{n_j-s}}\ge\frac{(\pi_D)^*\chi\cdot (F^{n_j})_*(\hat{R})}{C_l(n_j)d^{n_j-s}}= d^s S_j^{(0)}.
\end{equation}

  \begin{lemma}\label{l:sjh}
 For every $s$ and $j$ such that $n_j>s$ we
  have
 \begin{equation*}
      \|S_j^{(s)}\wedge (\pi_D)^*\omega_D^{h}\|_{D\times \G(p,k)}=\sum_{m=0}^{p-h}{p-h \choose m}\frac{C_{m+h}(n_j-s)}{C_l(n_j)}
      \quad
      \mbox{  for all }
      \quad
      0\le h\le p.
 \end{equation*}
 \end{lemma}
 
 \begin{proof}
 By a direct computation,
 for all $h$ as in the statement we have
 \begin{align*}
   \|S_j^{(s)}\wedge (\pi_D)^*\omega_D^{h}\|_{D\times \G(p,k)}= &\int_{D\times \G(p,k)} S_j^{(s)} \wedge (\pi_D)^*\omega_D^{h}\wedge ((\pi_D)^*(\omega_D)+(\pi_{\mathbb{G}})^*\omega_\G)^{p-h}\\  
   =&\sum_{m=0}^{p-h}{p-h \choose m}\int_{D\times \G(p,k)} {S_j^{(s)}\wedge (\pi_D)^*(\omega_D^{m+h})\wedge(\pi_\G)^*(\omega_\G^{p-h-m}) }\\
     =& \sum_{m=0}^{p-h}{p-h \choose m}\frac{I_{m+h}(n_j-s)}{C_l(n_j)d^{n_j-s}}=\sum_{m=0}^{p-h}{p-h \choose m}\frac{C_{m+h}(n_j-s)}{C_l(n_j)},
 \end{align*}
 where in the last step we used \eqref{e:CI}.
 The assertion follows.
\end{proof}

By the choice of $l$ and $\{n_j\}_{j\in \mathbb N}$ and Lemma \ref{l:subsequence}, for all $0\le m\le p$ we have $C_m(n_j-s)\lesssim C_l(n_j)$.
It follows from the choice of $\{n_j\}_{j\in \mathbb N}$ and Lemma
\ref{l:sjh} that
there exists a constant $c>1$ independent of $s$ and $j$ such that
\begin{equation*}
     1\le \|S_j^{(0)}\|_{D\times \G(p,k)}
     \quad
     \mbox{ and }\quad\|S_j^{(s)}\|_{D\times \G(p,k)}\le c \quad \mbox{ for all } s \geq 0.
\end{equation*}
In particular, 
for every
$s\geq 0$,
the sequence
$\{\|S_j^{(s)}\|\}_{j\in \mathbb N}$
is
uniformly bounded from above
by a constant independent of $s$.
We can then
fix in what follows 
a
current
$S^{(0)}_\infty$, which is a limit value
of the sequence $\{S_{j}^{(0)}\}_{j\in \mathbb N}$
along a given subsequence $\{j_i\}_{i\in \N}$, and a current
$S^{(s)}_\infty$, which is a limit value
of the sequence $\{S_{j}^{(s)}\}_{j\in \mathbb N}$
along a further subsequence of the sequence $\{j_i\}_{i\in \mathbb N}$.
By construction, we have
\begin{equation*}
     1\le \|S_\infty^{(0)}\|_{D\times \G(p,k)}
     \quad
     \mbox{ and }\quad\|S_\infty^{(s)}\|_{D\times \G(p,k)}\le c \quad \mbox{ for all } s\geq 0,
\end{equation*}
for a constant $c$ as above.

\begin{lemma}\label{l:c1}
The following properties hold for every $s\geq 0$:
\begin{enumerate}[label=(\roman*)]
     \item[{\rm (i)}]   
       $d^s S_\infty^{(0)}\le (F^s)_*(S^{(s)}_\infty)$;
           \item[{\rm (ii)}]  
             $S_\infty^{(s)}$ is
       compactly supported in ${D'}\times\G(p,k)$ and is
       closed on  $M''\times N\times\G(p,k)$;
       \item[{\rm (iii)}]   
            $S_\infty^{(s)}\wedge (\pi_D)^*\theta=0$ for any smooth $(l+1,l+1)$-form $\theta$ on $D$; 
         \item[{\rm (iv)}]  
            $\|S_\infty^{(0)}\wedge (\pi_D)^*(\omega_{D}^l)\|_{D'\times\G(p,k)}=1$ and there exists a constant 
         $c>0$ independent of $s$ such that $\|S_\infty^{(s)}\wedge (\pi_D)^* (\omega_{D}^l)\|_{D'\times\G(p,k)}\le c$. 
      \end{enumerate}
\end{lemma}

 \begin{proof}
It is immediate to see that 
(i)
follows from \eqref{eq:fsgedss}.   Since $f$ is horizontal-like,
$R$ has horizontal support in $M\times N'$, and $\chi$ has vertical support in $M'\times N$ it follows that 
$S_j^{(s)}$ is supported on $D'\times\G(p,k)$ for every $s\in \mathbb N$. 
Moreover,  since by \eqref{e:hp-inclusion} 
the current $f_* (R)$ 
is supported on
$M\times N''$, for every $j\geq 1$ 
the current $S_j^{(s)}$
is supported on
$\left(\mathrm{supp}(\chi)\cap (M\times N'' ) \right)\times \G (p,k) \Subset {D'}\times\G(p,k) $.
Hence,
$S^{(s)}_\infty$ has
compact support in $D'\times\G(p,k)$. Since by definition $S_j^{(s)}$ is closed in a neighbourhood of $\overline{M''}\times N\times\G(p,k)$
(as $\chi$ is constantly equal to 1 on an open neighbourhood of such set), 
the currents
$S_\infty^{(s)}$ are closed in  $M''\times N\times\G(p,k)$.  So, (ii)
follows.

\medskip

   In order to prove (iii),
     $S_\infty^{(s)}$ 
   is positive, it is enough to show that 
   \begin{equation}\label{eq:ssomegaD=0}
       S_\infty^{(s)}\wedge (\pi_D)^*\omega_D^{l+1}=0.
   \end{equation}
 Since $S_j^{(s)}\wedge (\pi_D)^*\omega_D^{l+1}$ is positive, it suffices to show that its mass tends to zero as $j\to\infty$.    By taking $h=l+1$ in Lemma \ref{l:sjh}
 (recall that $l\leq p-1$)
 and the fact that $l$ was given by  Lemma \ref{l:subsequence}
 (and in particular 
 the quantities $C_{\tilde l}$
 as in \eqref{e:CI} satisfy Lemma \ref{l:subsequence} (ii))
 we have
   \begin{align*}
     \lim_{j\to \infty}\|S_j^{(s)}\wedge (\pi_D)^*\omega_D^{l+1}\|& =  \lim_{j\to \infty} \sum_{m=0}^{p-l-1} {p-l-1 \choose m}\frac{C_{l+1+m}(n_j-s)}{C_l(n_j)} =0. 
   \end{align*}
  This gives \eqref{eq:ssomegaD=0}. The assertion
  (iii)
   follows.
  \medskip
  
  By taking  $h=l$ in Lemma 
  \ref{l:sjh} and using Lemma \ref{l:subsequence} (ii) again  we have
  \begin{align*}
     \lim_{j\to \infty}\|S_j^{(0)}\wedge (\pi_D)^*\omega_D^{l}\|& = 1+ \lim_{j\to \infty} \sum_{m=1}^{p-l} {p-l \choose m}\frac{C_{l+m}(n_j)}{C_l(n_j)} =1. 
   \end{align*}
By
taking  $h=l$ in Lemma 
  {\ref{l:sjh}} 
  and using Lemma \ref{l:subsequence} (i) we have
  \begin{align*}
     \|S_j^{(s)}\wedge (\pi_D)^*\omega_D^{l}\|& =\frac{C_{l}(n_j-s)}{C_l(n_j)} +  \sum_{m=1}^{p-l} {p-l \choose h}\frac{C_{l+m}(n_j-s)}{C_l(n_j)} \le c 
   \end{align*}
for some constant 
$c$ independent of $s$ and $j$.  This gives 
(iv)
and completes the proof. 
\end{proof}
 By Lemma \ref{l:c1}, both
$S^{(s)}_\infty$ and $S^{(0)}_\infty$ satisfy
the conditions in Definition \ref{d:D-dimension}, with $l_0=l$.
\begin{lemma}\label{l:ddcSs=0}
 Let $\theta $ be a smooth $(l,l)$-form 
   compactly supported 
   in $D$. Then we have
$$\partial
\left( (F^s)_*(S_\infty^{(s)})
       \wedge
       (\pi_D)^*\theta\right)=0
       \quad \mbox{ for every } s\geq 0.$$
\end{lemma}
\begin{proof}
Let us first show that 
the equality
\begin{equation}\label{eq:for-ddc-1}
(\pi_D)^*{\chi}
{\partial}
S_\infty^{(s)}=(\pi_D)^*({ \partial}\chi)\wedge S_\infty^{(s)}
\end{equation}
holds for every $s\geq 0$.
    By the definition of $S_j^{(s)}$ 
    and the fact that
    $ (F^{n_j-s})_*(\hat{R})$ is closed
    we have
$${ \partial} S_{j}^{(s)}
=\frac{(\pi_D)^*( {\partial} \chi) \cdot (F^{n_j-s})_*(\hat{R})}{C_l(n_j)d^{n_j-s}}$$
for every $s\geq 0$ and every $j$ such that $n_j>s$.
Multiplying $(\pi_D)^*\chi$ to both sides we obtain
$$(\pi_D)^*(\chi) { \partial} S_{j}^{(s)}
=(\pi_D)^*( { \partial} \chi)\wedge \frac{ (\pi_D)^*\chi\cdot (F^{n_j-s})_*(\hat{R})}{C_l(n_j)d^{n_j-s}}=(\pi_D)^*( { \partial} \chi)\wedge S_{j}^{(s)},$$
where we used the fact that the support of $ \partial \chi$ is contained in the support of $\chi$.
\eqref{eq:for-ddc-1} follows taking $j\to \infty$.

It 
follows from
\eqref{eq:for-ddc-1}
that, for any $\theta$
as in the statement and every $s\geq 0$,
we have
\begin{equation}\label{eq:for-ddc-3}
    (\pi_D)^*{\chi}
    {\partial}\left( S_\infty^{(s)}
       \wedge
       (\pi_D)^*((f^s)^*\theta)\right) 
       =S_\infty^{(s)}\wedge (\pi_D)^*(\Theta),
    \end{equation}
       where
 \[
\Theta   := 
       { \partial} \chi
       \wedge(f^s)^*\theta
       + \chi\cdot  (f^s)^*( { \partial} \theta)
          \]
 is a smooth 
 vertical
 $(l+1,{l})$-form  
in $D$.
Since  the $D$-dimension of $S_\infty^{(s)}$ 
is equal to $l$,
{by Lemma \ref{l:CS}} 
the right hand side of \eqref{eq:for-ddc-3} is equal to $0$.
As $\mathrm{Supp} \,S_\infty^{(s)}\subseteq \mathrm{Supp} \left((\pi_D)^*\chi\right)$, this implies
${ \partial} \left( S_\infty^{(s)}  \wedge    (\pi_D)^*((f^s)^*\theta)\right)=0$. So, we have
$${ \partial} \left( (F^s)_*(S_\infty^{(s)})        \wedge
       (\pi_D)^*\theta\right)=(F^s)_*
       { \partial}\left( S_\infty^{(s)}  \wedge    (\pi_D)^*((f^s)^*\theta)\right)=0.$$ The proof is complete.
\end{proof}

Let  now 
$\underline{S}^{(s)}_\infty$ and $\underline{S}_\infty^{(0)}$
be the shadows of 
${S}^{(s)}_\infty$ and
${S}_\infty^{(0)}$, respectively.
By Proposition \ref{p:shadow}, these
are positive currents of
bidimension $(l,l)$ on $D$.

\begin{lemma}\label{l:c2}
The following properties hold for every $s\geq 0$:
\begin{itemize}
\item[{\rm (i)}]   
we have $\|\underline{S}^{(0)}_\infty\|_{D'}=1$; furthermore, there exists a constant $c>0$ independent of $s$ such that $\|\underline{S}^{(s)}_\infty\|_{D'}\le c$;
    \item[{\rm (ii)}]   
        $\underline{S}^{(s)}_\infty$
     is
      horizontal on $M\times N'$ and closed on  $M''\times N$;
    \item[{\rm (iii)}]   
      $d^s \underline{S}^{(0)}_\infty\le (f^s)_*(\underline{S}^{(s)}_\infty)$.
  \end{itemize}
  \end{lemma}
\begin{proof}
Assertion (i)
follows from Proposition \ref{p:shadow} (applied with $D'$ instead of $\tilde D$)
and Lemma \ref{l:c1} 
(iv).
Assertion
(ii)
follows from 
Lemma \ref{l:c1}
(ii).
  Hence, we only have to prove
    the assertion 
   (iii).
  Fix $s\geq 0$ and let $\theta $ be a positive $(l,l)$-form 
   compactly supported 
   in $D$. We are going to show that
\begin{equation}\label{e:it:ineq}
 \langle \underline{S}^{(0)}_\infty,\theta \rangle
\leq 
\langle (f^s)_*(\underline{S}_\infty^{(s)}),\theta \rangle.
\end{equation}
   
\medskip

Recall that
$F$ can also be seen as a map from
$D_{v,1} \times \mathbb{P}^{L(p, k)}$
to
$D_{h,1} \times \mathbb{P}^{L(p, k)}$, and that
$\omega_{\G}$ is the form induced on $\G(p,k)$
by the 
Fubini-Study form on $\mathbb P^L= \mathbb P^{L(p,k)}$.
For every 
$z_0\in D_{v,1}$ we can identify the second coordinate of $F^s(z_0,\cdot)$ as an automorphism
 of $\P^{L}$.
By the invariance of the class of $\omega_{\P^{L}}$ 
 under the actions
 of such automorphisms,
  we have   
 \[(F^s)_*( (\pi_{\P^L})^*
 { \omega_{\P^L}})=(\pi_{\P^L})^*{\omega_{\P^L}}
 + dd^c {u_s}+ v_s,\]
for some   smooth $2$-form {$v_s$} which vanishes on the fibres 
of $\pi_{D}$ and
smooth function $u_s$,
on the common
 domain of definition of $(F^s)_*( (\pi_{\P^L})^*\omega_{\P^L})$ and 
 $(\pi_{\P^L})^*\omega_{\P^L}$. 
  It follows that, on the same set, we have 
  \begin{equation}\label{eq:Fomega}
  (F^s)_*( (\pi_{\P^L})^*
\omega_{\P^L}^{p-l})=(\pi_{\P^L})^* \omega_{\P^L}^{p-l}
 + dd^c U_s + V_s,
 \end{equation}
for some form $U_s$ and 
a  form $V_s$
which vanishes on the fibres 
of  $\pi_{D}$.
Hence, $V_s$ can be written as a linear combination of forms of the form  $(\pi_{D})^*\alpha \wedge V'_s$, where $\alpha$ is a $1$-form on $D$.

By using
  \eqref{eq:defshadow},
Lemma \ref{l:c1} 
(i),
 and
 \eqref{eq:Fomega},
  we have
   \begin{align*}
       \langle \underline{S}^{(0)}_\infty,\theta \rangle 
       &      =\int\limits_{D\times\G(p,k)}S_\infty^{(0)}\wedge (\pi_{\mathbb{G}})^*\omega_{\G}^{p-l}\wedge (\pi_D)^*\theta \\
       &
       \le
       d^{-s}\int\limits_{D\times\P^L}
        (F^s)_*(S_\infty^{(s)})
       \wedge (\pi_{\mathbb{P}^L})^*\omega_{\P^L}^{p-l}\wedge (\pi_D)^*\theta\\
   & 
       =
       d^{-s}\int\limits_{D\times\P^L}
              (F^s)_*(S_\infty^{(s)})
       \wedge
       (F^s)_* 
       ((\pi_{\mathbb{P}^L})^*\omega_{\P^L}^{p-l})
       \wedge (\pi_D)^*\theta \\
        & \quad -
        d^{-s}\int\limits_{D\times\P^L}
              (F^s)_*(S_\infty^{(s)})
       \wedge
      dd^c U_s
             \wedge (\pi_D)^*\theta\\
       & \quad -
        d^{-s}\int\limits_{D\times\P^L}
             (F^s)_*(S_\infty^{(s)})
       \wedge
             {V_s}
       \wedge (\pi_D)^*\theta.
\end{align*}
 
The second integral in the last term vanishes
by Stokes' formula and
Lemma \ref{l:ddcSs=0}. We claim that
the third integral also vanishes.
Indeed, by the above description of $V_s$, it
is enough to show this claim for $V_s$ of the form
$(\pi_D)^*\alpha \wedge V'_s$, where $\alpha$ 
is a $1$-form on $D$.
This implies that
$(\pi_{D})^*(\alpha\wedge\theta)$ is  $(2l+1)$-form and thus Lemma \ref{l:CS}  implies that $$(F^s)_*(S_\infty^{(s)}) \wedge (\pi_{D})^*(\alpha\wedge\theta)=(F^s)_*(S_\infty^{(s)} \wedge (\pi_{D})^*((f^s)^*(\alpha\wedge\theta)))=0.$$
As $F^s$ is invertible, this implies
\begin{align*}
          \langle \underline{S}^{(0)}_\infty,\theta \rangle 
          &
          \leq 
          d^{-s}\int\limits_{D\times\P^L}
                   (F^s)_*\left(S_\infty^{(s)}\wedge (\pi_{\mathbb{P}^L})^*\omega_{\P^L}^{p-l}\right)\wedge (\pi_D)^*\theta \\ 
           &=d^{-s}\int\limits_{D\times\G(p,k)} S_\infty^{(s)}\wedge (\pi_{\mathbb{G}})^*\omega_{\G}^{p-l}\wedge (F^s)^*( (\pi_D)^*\theta) \\ 
           &=d^{-s} \langle \underline{S}_\infty^{(s)},(f^s)^*(\theta) \rangle=d^{-s} \langle (f^s)_*(\underline{S}_\infty^{(s)}),\theta \rangle. 
   \end{align*}
This shows 
\eqref{e:it:ineq}
and completes the proof.
\end{proof}

We can now complete the proof of Proposition \ref{p:In=Odn}.

\begin{proof}[End of the proof of Proposition \ref{p:In=Odn}]
We continue to use the notations introduced above. 
By Lemma \ref{l:c2} 
(i) and (ii),
we have $\|\underline{S}^{(0)}_\infty\|_{M'\times N}=1$.
By
Lemma \ref{l:c2} (iii),
this gives $$d^s=d^s\|\underline{S}_\infty^{(0)}\|_{M'\times N}\le\|(f^s)_*(\underline{S}^{(s)}_\infty\|_{M'\times N}$$
for every $s\geq 0$.
Since $\underline{S}^{(s)}_\infty$ is a closed horizontal current on $M''\times N$, by
\eqref{e:hp-inclusion} we 
see that ${f_*(\underline{S}^{(s)}_\infty)}$ is a closed horizontal current on $M'\times N$. Moreover, since the mass of ${\underline{S}^{(s)}_\infty}$  is uniformly bounded from above
by a constant independent of $s$,
the mass of ${f_*(\underline{S}^{(s)}_\infty})$ is also bounded by
a constant independent of $s$.
By using the Definition \ref{d:degrees-currents}
of the dynamical degrees,
it  follows that
$$d\le \limsup_{s\to \infty}\|(f^{s})_*(f_*(\underline{S}^{(s+1)}_\infty))\|_{M'\times N}^{1/s}\le d^+_l.$$
On the other hand, since $l<p$, by \cite[Theorem 1.1]{BDR}
and the assumption $d^+_{p-1}< d$
we have $d_l^+\le d^+_{p-1}< d$. This gives the desired contradiction. So, we have $I(n)=O(d^n)$
and the proof is complete.
\end{proof}

\end{document}